\documentclass[a4paper,12pt]{article}
\usepackage[utf8]{inputenc}
\usepackage[english]{babel}
\usepackage{graphicx}
\usepackage{amsthm, amsmath, amssymb, amsfonts}
\usepackage[all]{xy}
\usepackage{hyperref} 
\usepackage[top=2.5cm, bottom=2.5cm, left=2.5cm, right=2.5cm]{geometry}
\usepackage{tikz}
\usetikzlibrary{matrix}
\usepackage{lscape}
\usepackage{mathrsfs}
\usepackage{dsfont}
\usepackage{xcolor}
\usepackage{cancel}
\usepackage{comment}
\usepackage{upgreek}
\usepackage{mathtools}
\usepackage{tikz}
\usepackage{tkz-tab}%
\usepackage{subcaption}%
\usepackage{caption}%
\usepackage{ucs}%
\usepackage{pgfplots}

\newcommand\Ccancel[2][black]{\renewcommand\CancelColor{\color{#1}}\cancel{#2}}

\usepackage[active]{srcltx}
\usepackage{color}

\newtheorem{theorem}{Theorem}[section]
\newtheorem{lemma}{Lemma}[section]
\newtheorem{corollary}{Corollary}[section]
\newtheorem{conjecture}{Conjecture}[section]
\newtheorem{proposition}{Proposition}[section]
\theoremstyle{definition}
\newtheorem{definition}[theorem]{Definition}

\newtheorem{remark}{Remark}[section]

\newtheorem{thmx}{Theorem} 
\newtheorem{corox}{Corollary}[thmx] 
\usepackage{makeidx}
\makeindex

\title{Partially hyperbolic geodesic flow via conformal deformation}
\author{Ygor de Jesus\footnote{Ygor de Jesus was financially supported by Fundação de Amparo à Pesquisa do Estado de São Paulo (FAPESP - Brazil) through process number 2021/02913-0 and 2023/05100-5.}, Luis Pedro Piñeyrúa\footnote{Luis Pedro Piñeyrúa was supported by Instituto Serrapilheira, grant F0265 “Jangada Dinâmica: Impulsionando Sistemas Dinâmicos na Região Nordeste”.}, Sergio Romaña\footnote{ Sergio Roma\~na was supported by NSFC grant number W2632010 from China.}}
\date{}

\newcommand{\til}{\widetilde}

\newcommand{\wh}{\widehat}

\newcommand{\cPH}{{\mathcal{PH}}}

\hypersetup{
plainpages=false,
colorlinks,
linkcolor={blue!100!black},
citecolor={blue!100!black},
urlcolor={blue!100!black},
hypertexnames=true,
bookmarksnumbered,
}

\newcommand{\R}{\mathbb{R}}

\newcommand{\mc}{\mathcal}

\newcommand{\mrm}{\mathrm}

\newcommand{\norm}[1]{\left\lVert#1\right\rVert}
\newcommand{\ve}{\varepsilon}
\newcommand{\rarrow}{\rightarrow}

\newcommand*{\Scale}[2][4]{\scalebox{#1}{$#2$}}%

\newcommand{\Addresses}{{
  \bigskip
  \footnotesize

Y. de Jesus, \textsc{Departamento de Matem\'atica, IMECC-UNICAMP Campinas-SP, Brazil.} \par\nopagebreak
  \textit{E-mail}: \texttt{y157606@dac.unicamp.br}

L. P. Piñeyrúa, \textsc{IMERL, Facultad de Ingeniería
Universidad de la República, Montevideo, Uruguay.} \par\nopagebreak
  \textit{E-mail}: \texttt{lpineyrua@fing.edu.uy}

S. Romaña, \textsc{School of  Mathematics (Zhuhai), Sun Yat-sen University, 519082, PR China.} \par\nopagebreak
  \textit{E-mail}: \texttt{sergio@mail.sysu.edu.cn}

}}

\begin{document}
\maketitle
\begin{abstract}
   This paper presents a new construction of non-Anosov partially hyperbolic geodesic flows. Our construction is closely related to the construction presented in \cite{CARNEIRO&PUJALS14}, the novelty is the use of conformal deformations to produce the examples. Some of the necessary conditions appear more naturally and are easier to check. Moreover, we can enumerate the conditions required to produce partially hyperbolic geodesic flow examples. We show how to produce examples with metrics that are non-positively curved, and with a finer analysis we can prove ergodicity for the Liouville measure and uniqueness of the measure of maximal entropy. These examples lie on the boundary of Anosov metrics, allowing us also to produce metrics with partially hyperbolic geodesic flows and conjugate points.
 
\end{abstract}
\tableofcontents

\section{Introduction}
Given a closed, connected Riemannian manifold $(M,g)$, one of the oldest open problems in dynamical systems and ergodic theory is to determine conditions that guarantee ergodicity of the geodesic flow with respect to the \textit{Liouville measure}. A classical result states that the geodesic flow is of Anosov type when all the sectional curvatures are strictly negative. This connection is indeed the first model of Anosov systems and the origin of the well-known technique called \textit{Hopf's argument}, which is very useful for obtaining ergodicity. It is still not a completely solved question whether the geodesic flow for a surface of genus $g\geq 2$ and non-positive curvature is ergodic with respect to the Liouville measure. However, Pesin obtained great advances in \cite{PESIN1977geodesic}. In this work, Pesin proves that if the action of the geodesic flow is restricted to vectors for which the curvature is negative along their underlying geodesic, then this flow is conjugate to a Bernoulli flow, and in particular, it is ergodic. 

The analogous question in higher dimensions is whether non-positively curved Riemannian metrics of rank one are ergodic. The rank of a tangent vector $v\in T_xM$ is the dimension of the space of parallel orthogonal \textit{Jacobi fields} along the geodesic $\gamma_{(x,v)}(t)$, and the rank of the manifold $(M,g)$ is defined as the infimum of the rank of the unit tangent vector. Celebrated works from the 80s such as \cite{BALMANN&BRIN1982ergodicity} and \cite{BURNS1983hyperbolic} explored the relation between the ergodicity of the geodesic flow and the rank of the metric. They concluded that the restriction of the geodesic flow to the set of rank one vectors is also Bernoulli. 

It is natural to question if non-positivity of the curvature is essential to obtain ergodicity for the geodesic flow. In \cite{BURNS1989S2} and \cite{DONNAY2006ERGODICITY} a negative answer to this question is given. They obtain examples of ergodic geodesic flows on surfaces with some positive curvature. It is worth mentioning that in their case the presence of negative curvature is still essential to obtaining ergodicity. It is also not true that the Anosov property for the geodesic flow must imply negative curvature; indeed, in \cite{EBERLEIN73}, Eberlein states several equivalent conditions to the Anosov property. Essentially, negative curvature must appear along any geodesic, but it does not prevent positive curvature from appearing. Also, in \cite{GULLIVER1975}, Gulliver constructs examples of Anosov geodesic flows such that the underlying metric has some positive sectional curvatures.

The time$-1$ map of an Anosov geodesic flow is one of the most classic examples of partially hyperbolic systems. By partial hyperbolicity we mean that there is a continuous splitting of the tangent bundle $TM=E^s\oplus E^c\oplus E^u$ into three non-trivial invariant subbundles, where $E^s$ is uniformly contracted for the future, $E^u$ is uniformly contracted for the past, and $E^c$ has an intermediate behavior, i.e., it does not contract nor expand as much as the other two (see Section \ref{Preliminaries} for the precise definition). It is a famous conjecture by Pugh and Shub that ergodic systems inside this class (see \cite{RHTU2011CRITERIA}, \cite{BURNS&RH2008} and \cite{BURNS&WILKINSON2010}). According to the Pugh and Shub program, this main conjecture may be a consequence of two conjectures: accessibility is dense among partially hyperbolic systems, and accessibility implies ergodicity.

It is well known that geodesic flows are examples of contact flows, i.e., the geodesic vector field is the Reeb flow of the contact form induced by the Riemannian metric. In \cite{CARNEIRO&PUJALS14}, Carneiro and Pujals constructed the first class of examples of partially hyperbolic geodesic flows and, in \cite{FISHER&HASSELBLATT2022accessibility}, Fisher and Hasselblatt proved that it is possible to perturb any contact form for which the Reeb flow is partially hyperbolic to make it accessible. However, there is a lack of examples in this class of systems and one of the consequences of this article is to present new examples and the study of new techniques to produce them. Besides the genericity of the ergodic property in some classes inside the partially hyperbolic setting, it is a delicate problem to produce examples of ergodic systems that are not of Anosov type. It is even more delicate to produce such examples in the class of geodesic flows. In the two-dimensional setting, there are simple constructions one can think of, for example, the surface of revolution obtained by rotating the graph of $f(x)=x^4+1$, defined in some interval $[-a,a]$, around the $x$-axis and then gluing two negatively curved surfaces to this neck. The obtained surface is non-positively curved, and moreover it has negative Gaussian curvature besides a central closed geodesic at $x=0$ for which the Gaussian curvature is identically zero. Such an example has a Bernoulli geodesic flow that is not of Anosov type, by the previously mentioned works of Pesin and Eberlein. This example also presents more interesting dynamical properties, for example, the decay of correlations studied in \cite{LIMA&MATHEUS2024}. One quickly realizes that this construction of examples by rotating the graph of a function cannot be generalized to higher dimensions to obtain non-positively curved manifolds with a single geodesic with vanishing sectional curvatures, hence different techniques and approaches are needed.

In \cite{RUGGIERO1991creation} Ruggiero proved that the $C^2$-interior of the set of Riemannian metrics without conjugate points, say $\mathcal{NC}(M)$, coincides with the set of Riemannian metrics for which the geodesic flow is of Anosov type, say $\mathcal{A}(M)$. It is still not well understood what the set of metrics on the boundary of $\mathcal{NC}(M)$ is, nor the relation to the expansivity of the geodesic flow. We say that a continuous flow is expansive if two different orbits eventually grow apart. It is known in the two-dimensional setting that expansivity of the geodesic flow implies the absence of conjugate points (see \cite{PATERNAIN1993expansive}); however, it is a conjecture attributed to Ricardo Mañe that this result should be true for higher dimensions as well. Notice that the converse is not true, since we can consider a piece of flat cylinder and glue two negatively curved surfaces to the boundary of this flat neck. For this construction, we have that for every $\varepsilon>0$, there exist two closed geodesics $\gamma_1(t)$ and $\gamma_2(t)$ in the cylinder that are $\varepsilon$-close for every $t\in \mathbb{R}$. This phenomenon is detected by the \textit{Flat strip theorem} which states that if the geodesic flow is not expansive, then for every $\varepsilon>0$ there exist two distinct geodesics $\gamma_1$ and $\gamma_2$ such that $d(\gamma_1(t),\gamma_2(t))<\ve$, for all $t\in \R$, thus we can find a flat strip on the universal cover, i.e. an isometric copy of $[-\delta,\delta]\times \R$. Therefore, to produce expansive geodesic flows that are not of Anosov type, it is necessary to have some control on the amount of zero curvature in the non-positively curved setting. It is interesting to remark that in \cite{RUGGIERO1997expansive} Ruggiero proves that expansive geodesic flows with no conjugate points have density of periodic points and a local product structure, which implies they are topologically transitive, i.e. they admit a dense forward orbit. Hence, examples of such geodesic flows are rich from both the geometric and dynamical point of view.

Just like the ergodicity in the non-Anosov case, there is a shortage of higher-dimensional examples of expansive geodesic flows admitting some zero sectional curvatures. A natural attempt would be to start with two negatively curved manifolds, say $(M_1,g_1)$ and $(M_2,g_2)$, and then consider the Riemannian product manifold $(M_1\times M_2,g_1+g_2)$. In this case, the product metric admits several vanishing sectional curvatures, and it is not even partially hyperbolic, as proved in \cite{CARNEIRO&PUJALS14}. Therefore, it is crucial to develop techniques of metric deformation. By this, we mean the construction of examples by changing the curvature of a negatively curved metric in a controllable way. 
In this article, we explore the use of conformal deformation in order to produce the above-mentioned examples. We say that two Riemannian metrics $g$ and $g^*$ are $C^k$-conformally related if there exists a positive function $\phi:M\to \R_{>0}$ of class $C^k$ such that $g^*=\phi g$. It is clear that the conformality of Riemannian metrics is an equivalence relation; thus we can consider the conformal class of a metric. In \cite{KATOK1982ENTROPY}, \cite{KATOK1988CONFORMAL} and \cite{BARTHELME&ERCHENKO2021}, Katok, Barthelmé, and Erchenko study rigidity and flexibility on the dynamics inside the conformal class of a Riemannian metric; namely, they study the variation of the entropy and length spectra for some conformally related metrics. This technique of deformation was also used by Ruggiero in the previously mentioned paper \cite{RUGGIERO1991creation}. We use the analysis of conformally related Riemannian metrics to produce the following result:

\begin{thmx}
\label{there exist conformal partially hyperbolic}
Let $(M,g)$ be a compact Kähler manifold of holomorphic curvature $-1$ or a compact locally symmetric quaternionic Kähler manifold of negative curvature. Then there exists a conformal metric $\til g=\phi g$ such that the geodesic flow $\til g_t$ is partially hyperbolic and not Anosov.
\end{thmx}

The Riemannian manifolds considered in this result are contained in a larger class of interesting Riemannian manifolds called \textit{locally symmetric manifolds} (it is also common to just say locally symmetric metric since it is a metric property). By definition, a metric is said to be locally symmetric if its curvature tensor is parallel with respect to the \textit{Levi-Civita connection}, i.e., $\nabla R\equiv 0$. One of the interesting properties of this class of metrics, from the geometry point of view, is the fact that the sectional curvatures $K$ are not affected by parallel transport. This means that given two tangent vectors $X,Y$, if $\tilde{X}(t)$ and $\tilde{Y}(t)$ denote their parallel transport along a geodesic, then $K(\tilde{X},\tilde{Y})(t)=K(\tilde{X},\tilde{Y})(0)=K(X,Y)$. From the dynamical point of view, since we are considering negatively curved metrics, the corresponding geodesic flow is of Anosov type. However, it also has a partially hyperbolic structure on the contact structure, i.e., the splitting occurs not only into two invariant subbundles but into four invariant subbundles in a way that we have stronger and weaker behavior (see Lemma \ref{four invariant bundles}). Besides this splitting property, the class of locally symmetric Riemannian metrics also has an interesting behavior with respect to the topological entropy of its geodesic flow. In \cite{KATOK1982ENTROPY}, Katok proved that for surfaces of genus greater than or equal to two we have a fascinating comparison between the measure entropy for the Liouville measure and topological entropy of different Riemannian metrics. More precisely, let $g$ be any negatively curved Riemannian metric and $g^*$ be a Riemannian metric with constant negative curvature. Suppose in addition that $g$ and $g^*$ determine the same area, then $h_{\mathrm{Liou}}(g)\leq h_{\mathrm{Liou}}(g^*)$ and $h_{\mathrm{top}}(g)\geq h_{\mathrm{top}}(g^*)$. The inequalities are strict when $g$ has non-constant negative curvature. This result led Katok to state the following conjecture for negatively curved Riemannian metrics

\begin{conjecture}[\cite{KATOK1982ENTROPY}]
\label{conjecture katok 1}
    One has $h_{\mathrm{top}}(g)=h_{\mathrm{Liou}}(g)$ on a compact manifold if and only if $g$ is a locally symmetric Riemannian metric.
\end{conjecture}

Conjecture \ref{conjecture katok 1} was proved by Katok himself in \cite{KATOK1982ENTROPY} for metrics that are conformal to a locally symmetric one. In particular, it implies that the result is true for surfaces. It is known by \cite{KKM91} that locally symmetric metrics are critical points for both $h_{\mathrm{top}}$ and $h_{\mathrm{Liou}}$ when seen as functions from the space of Riemannian metrics. More about this conjecture can be found in \cite{BESSON95}. We are unable to relate the above conjecture to the conformal deformation construction in Theorem \ref{there exist conformal partially hyperbolic} since we do not know if the resulting metric is even non-positively curved; the best we can state is that if some positive curvature is created, then it can be made as small as we want. However, working in a slightly different and less rigid kind of deformation, we are able to produce an example with non-positive curvature. More specifically, 
\begin{thmx}
\label{ergodic partially hyperbolic geodesic flow}
    Let $(M,g)$ be a compact Kähler manifold of holomorphic curvature $-1$ or a compact locally symmetric quaternionic Kähler manifold of negative curvature, then there exists a $C^2$-deformation $\til g$ of the metric $g$ with the following properties
    \begin{enumerate}
        \item $\til g_t$ is partially hyperbolic.
        \item There exists a closed geodesic $\gamma$ with a parallel Jacobi field along it.
        \item The sectional curvatures $\til K$ are all negative outside $\gamma$.
    \end{enumerate}
\end{thmx}
As a consequence of the construction of Theorem \ref{ergodic partially hyperbolic geodesic flow} we get the following corollary:
\begin{corox}
\label{new geodesic flow has all the good properties}
    There exists a Riemannian manifold $(M,g)$ with no conjugate points such that its geodesic flow is partially hyperbolic, non-Anosov, ergodic for the Liouville measure, mixing, expansive, and has a unique measure of maximal entropy.
\end{corox}

We highlight the ergodicity for the Liouville measure, which was a question stated in \cite{CARNEIRO&PUJALS14}. It is somehow possible to relate the previous result to the aforementioned Katok's conjecture in the sense that we obtain a unique measure of maximal entropy. However, the metrics are non-positively curved and they are not negatively curved. Besides that, we are unable to say if the measure of maximal entropy is the Liouville measure or not. 

In \cite{CARNEIRO&PUJALS14}, the authors do not guarantee the existence or absence of conjugate points in their examples; indeed, they use a curve (a straight line) of metrics to prove the existence of a metric with partially hyperbolic geodesic flows and non-conjugate points. We were unable to find a reference that guarantees that the set of metrics for which the geodesic flow is partially hyperbolic is convex or at least path-connected. Thus the previous corollary guarantees, by other arguments, the absence of conjugate points for the resulting metric. Furthermore, we can also prove the existence of examples with conjugate points in the following way:  recall that Ruggiero proved in \cite{RUGGIERO1991creation} that the $C^2$-interior of the set of Riemannian metrics without conjugate points coincides with the set of Riemannian metrics for which the geodesic flow is of Anosov type. The metrics obtained via Theorem \ref{ergodic partially hyperbolic geodesic flow} are on the boundary of the set of Anosov Riemannian metrics, and their geodesic flows are partially hyperbolic. Since partial hyperbolicity is an open condition, it implies that there exists a $C^2$-open set of Riemannian metrics with conjugate points and partially hyperbolic, non-Anosov, geodesic flow. In other words, we obtain
\begin{corox}
There exists a metric $C^2$-close to $\tilde{g}$ that has conjugate points and a partially hyperbolic geodesic flow. Moreover, there exists a $C^2$-open set of Riemannian metrics with conjugate points for which the geodesic flow is partially hyperbolic.
\end{corox}

The paper is structured as follows: in Section \ref{Preliminaries} we present the necessary background from Riemannian geometry, partially hyperbolic dynamics, and the dynamics of the geodesic flow. For some elementary computations within Riemannian geometry, we were unable to find a reference, so we provided short proofs for completeness. In Section \ref{Breaking the Anosov Property}, we show how to use conformal deformations in order to destroy the Anosov property of the geodesic flow. The construction consists of deforming the initial metric $g$ by multiplying it by a conformal factor supported in a tubular neighborhood of a closed geodesic, which we will mention as \textit{central geodesic}. We investigate the necessary conditions on the multiplying factor and show how to construct it in general. The construction can be made with a multiplying function as regular as needed, at least of class $C^4$. It is also possible to make it smooth; however, with this technique, we cannot control how large the positive curvature that appears is. In Section \ref{The new geodesic flow is Partially Hyperbolic} we show that the resulting flow is partially hyperbolic by using the \textit{cone criterion} (see Section \ref{cone criteria}). We begin by explicitly exhibiting the invariant splitting along the central geodesic, which gives us the candidates for the invariant cone families. We compute the variation of the cone's openness along a general orbit of the flow, then the proof splits into some cases: for geodesics that are parallel to the central geodesic or almost parallel, we show that openness variation is approximated by a positive quantity and this implies cone invariance. We then show that we can shrink the deformed region in order to avoid transversal geodesics from losing hyperbolicity, and then we also obtain invariance of the cone family. The proof is completed by using the symplectic structure of the geodesic flow (see Lemma \ref{dominated splitting in symplectic}). This strategy is also used in \cite{CARNEIRO&PUJALS14}, however, the computations that appear are treated differently. Finally, in Section \ref{Proof of Theorem B} we prove Theorem \ref{ergodic partially hyperbolic geodesic flow} and Corollary \ref{new geodesic flow has all the good properties}. The proof of partial hyperbolicity with a different deformation works in the same way; thus we are left to analyze the curvature. We show that every sectional curvature can be bounded by a non-positive function, which vanishes just along the central geodesic.
\section{Preliminaries}
\label{Preliminaries}
Throughout the rest of this paper, $(M,g)$ will denote a compact Riemannian manifold without boundary and dimension $n> 2$.  We denote by $TM$ the tangent bundle and $T^1M$ its unit tangent bundle. For simplicity of notation, we are going to use the Einstein summation convention to avoid writing several summation symbols, i.e. we are going to write $x^iE_i$ to mean $\sum_{i}x^iE_i$ whenever the index $i$ appears once as an upper index and once as a lower index and the summation is considered among all possible values of $i$. We are going to use the summation sign whenever any confusion with indices is possible.
\subsection{Riemannian geometry}
\label{preliminaries riemannian geometry}
Great references for this section are \cite{LEE18riemannian} and \cite{WALSCHAP04}. Given a smooth connection $\nabla: \mathcal{X}(M)\times \mathcal{X}(M)\rightarrow \mathcal{X}(M)$ and a local frame $\{E_i\}$, the map $\nabla$ is determined by its action on this local frame as $\nabla_{E_i}E_j=\Gamma_{ij}^kE_k$, where $\{\Gamma_{ij}^k\}$ is a family of smooth functions defined wherever the local frame is defined. The functions $\Gamma_{ij}^k$ are called \textit{Christoffel symbols}. A classical result of Riemannian geometry is the existence of a unique smooth connection that is compatible with $g$ and symmetric, called the Levi-Civita connection, and its Christoffel symbols are given locally in terms of the metric components $g_{ij}$ and its inverse $g^{ij}$ by:
\[
\Gamma_{ij}^k=\frac{g^{kl}}{2}(\partial_ig_{lj}+\partial_jg_{li}-\partial_lg_{ij}).
\]
In this paper, we will consider the following convention for the curvature tensor
\[
R(X,Y)Z=\nabla_X\nabla_YZ-\nabla_Y\nabla_XZ-\nabla_{[X,Y]}Z.
\]
So, the sectional curvatures are given by
\[
K(X,Y)=\frac{g(R(X,Y)Y,X)}{|X\wedge Y|^2},
\]
where $|X\wedge Y|=\sqrt{g(X,X)g(Y,Y)-g(X,Y)^2}$.
For a conformal metric given by $\til{g}=\phi g$, where $\phi=\mrm{e}^h$, we have the following relations (see \cite{WALSCHAP04}) for the Christoffel symbols
\begin{equation}
    \label{conformal christoffel symbols}
    \til{\Gamma}_{ij}^k=\Gamma_{ij}^l+\frac{1}{2}(\partial_ih\delta_j^k+\partial_jh\delta_i^k-\partial_lhg^{lk}g_{ij}),
\end{equation}
for the covariant derivative given by the Levi-Civita connection
\begin{equation}
   \label{conformal connection}
    \til{\nabla}_XY=\nabla_XY+\frac{1}{2}(X(h)Y+Y(h)X-g(X,Y)\nabla h),
\end{equation}
and for the curvature tensor
\begin{align}
\label{conformal curvature}
    \til{R}(X,Y)Z=&R(X,Y)Z+\frac{1}{2}\{g(\nabla_X\nabla h,Z)Y-g(\nabla_Y\nabla h,Z)X\nonumber\\
    &+g(X,Z)\nabla_Y\nabla h -g(Y,Z)\nabla_X\nabla h\}\nonumber\\
    &+\frac{1}{4}\{((Yh)(Zh)-g(Y,Z)|\nabla h|^2)X-((Xh)(Zh)-g(X,Z)|\nabla h|^2)Y\nonumber\\
    &+((Xh)g(Y,Z)-(Yh)g(X,Z))\nabla h\}.
\end{align}
Observe that $\til{|X\wedge Y|}^2=\til{g}(X,X)\til{g}(Y,Y)-\til{g}(X,Y)^2=\phi^2 |X\wedge Y|^2$. Then, for any pair of linearly independent vectors $X, Y$, the sectional curvature is given by
\begin{align*}
    \phi\til{K}(X,Y)=&K(X,Y)+\frac{1}{2|X\wedge Y|}\{g(\nabla_X\nabla h,Y)g(Y,X)-g(\nabla_Y\nabla h,Y )g(X,X)\\
    &+g(X, Y)g(\nabla_Y\nabla h,X) 
    -g(Y,Y)g(\nabla_X\nabla h,X)\}\\
    &+\frac{1}{4}\{((Yh)(Yh)-g(Y,Y)|\nabla h|^2)g(X,X)\\
    &-((Xh)(Yh)-g(X,Y)|\nabla h|^2)g(Y,X)\\
    &+((Xh)g(Y,Y)-(Yh)g(X,Y))g(\nabla h,X)\}.
\end{align*}
The sectional curvature for an orthonormal basis $\{X,Y\}$ of a plane $P$ is given by
\begin{equation}
\label{conformal sectional curvature}
\phi \til{K}(X,Y)=K(X,Y)-\frac{1}{2}(g(\nabla_X\nabla h,X)+g(\nabla_Y\nabla h,Y))-\frac{1}{4}(|\nabla h|^2-(Xh)^2-(Yh)^2).  
\end{equation}
Here, $g(\nabla_X\nabla h,Y)$ is the so-called \textit{Riemannian hessian of the function $h$} and is also denoted by $\text{Hess}(h)(X,Y)$.
We are going to use the following elementary lemmas:
\begin{lemma}
\label{derivative of product metric by  inverse}
    If $g_{ij}$ are the components of a Riemannian metric in coordinates and $g^{ij}$ of its inverse, then the following expressions hold
    \begin{enumerate}
        \item \begin{equation}
            \label{derivative inverse}
            \partial_s g^{ik}g_{kj}=-g^{ik}\partial_sg_{kj}.
        \end{equation}
        \item \begin{equation}
            \label{derivative with christoffel}
            \partial_p g_{ms}=g_{ns}\Gamma_{pm}^n+g_{nm}\Gamma_{ps}^n.
        \end{equation}
    \end{enumerate}
\end{lemma}
\begin{lemma}
 \label{curvature in coordinates}
In any coordinates, the curvature tensor can be expressed as
    \begin{equation}
    R_{ijks}=R_{ijk}^lg_{ls}=-\frac{1}{2}(\partial_{js}^2g_{ik}+\partial_{ik}^2g_{js}-\partial_{is}^2g_{jk}-\partial_{jk}^2g_{is})-g_{mn}(\Gamma_{ki}^m\Gamma_{js}^n-\Gamma_{si}^m\Gamma_{jk}^n).
\end{equation}
\end{lemma}
\begin{proof}
    To be careful, we will analyze each part of the sum we are interested in. By using the previous Lemma,
    \begin{align*}
        g_{ls}\partial_j\Gamma_{ik}^l&=\frac{1}{2}(g_{ls}\partial_j g^{ml}(\partial_kg_{mi}+\partial_ig_{mk}-\partial_mg_{ik})+g_{ls}g^{ml}(\partial_{jk}^2g_{mi}+\partial_{ji}^2g_{mk}-\partial_{jm}^2g_{ik}))\\
        &=\frac{1}{2}(-g^{ml}\partial_j g_{ls}(\partial_kg_{mi}+\partial_ig_{mk}-\partial_mg_{ik})+\delta^m_s(\partial_{jk}^2g_{mi}+\partial_{ji}^2g_{mk}-\partial_{jm}^2g_{ik}))\\
        &=-\partial_jg_{ls}\Gamma_{ik}^l+\frac{1}{2}(\partial_{jk}^2g_{si}+\partial_{ji}^2g_{sk}-\partial_{js}^2g_{ik})\\
        &=-(g_{ns}\Gamma_{jl}^n+g_{nl}\Gamma_{js}^n)\Gamma_{ik}^l+\frac{1}{2}(\partial_{jk}^2g_{si}+\partial_{ji}^2g_{sk}-\partial_{js}^2g_{ik})\\
        &=\frac{1}{2}(\partial_{jk}^2g_{si}+\partial_{ji}^2g_{sk}-\partial_{js}^2g_{ik})-g_{ns}\Gamma_{jl}^n\Gamma_{ik}^l-g_{nl}\Gamma_{js}^n\Gamma_{ik}^l.
    \end{align*}
    By an analogous calculation, we have
    \[
    g_{ls}\partial_i\Gamma_{jk}^l=\frac{1}{2}(\partial_{ik}^2g_{sj}+\partial_{ij}^2g_{sk}-\partial_{is}^2g_{jk})-g_{ns}\Gamma_{il}^n\Gamma_{jk}^l-g_{nl}\Gamma_{is}^n\Gamma_{jk}^l.
    \]
        Subtracting the last two equations, we get,
    \begin{align*}
        g_{ls}\partial_j\Gamma_{ik}^l-g_{ls}\partial_i\Gamma_{jk}^l&=\frac{1}{2}(\partial_{jk}^2g_{si}+\partial_{is}^2g_{jk}-\partial_{js}^2g_{ik}-\partial_{ik}^2g_{sj})\\
        &+g_{ns}\Gamma_{il}^n\Gamma_{jk}^l+g_{nl}\Gamma_{is}^n\Gamma_{jk}^l-g_{ns}\Gamma_{jl}^n\Gamma_{ik}^l-g_{nl}\Gamma_{js}^n\Gamma_{ik}^l.
    \end{align*}
    Finally, $R_{ijks}$ is given by
    \begin{align*}
        R_{ijks}&=g_{ls}\partial_j\Gamma_{ik}^l-g_{ls}\partial_i\Gamma_{jk}^l+g_{ls}\Gamma_{ik}^r\Gamma_{jr}^l-g_{ls}\Gamma_{jk}^r\Gamma_{ir}^l\\
        &=\frac{1}{2}(\partial_{jk}^2g_{si}+\partial_{is}^2g_{jk}-\partial_{js}^2g_{ik}-\partial_{ik}^2g_{sj})\\
        &\Ccancel[blue]{+g_{ns}\Gamma_{il}^n\Gamma_{jk}^l}+g_{nl}\Gamma_{is}^n\Gamma_{jk}^l\Ccancel[red]{-g_{ns}\Gamma_{jl}^n\Gamma_{ik}^l}-g_{nl}\Gamma_{js}^n\Gamma_{ik}^l\\
        &\Ccancel[red]{+g_{ls}\Gamma_{ik}^r\Gamma_{jr}^l}-\Ccancel[blue]{g_{ls}\Gamma_{jk}^r\Gamma_{ir}^l}.
    \end{align*}
\end{proof}
\noindent We will work with the same manifolds considered in \cite{CARNEIRO&PUJALS14}:
\begin{enumerate}
    \item Compact Kähler manifolds of holomorphic curvature $-1$ (cf. \cite{GOLDMAN1999complex}).
    \item Compact locally symmetric quaternionic Kähler manifolds of negative curvature (cf. \cite{BESSE2007einstein}).
\end{enumerate}
 For the above manifolds, we consider the normalization of the metric such that their sectional curvatures are $\frac{1}{4}$-pinched, i.e. $-1\leq K\leq -\frac{1}{4}$. For every $(x,v)\in TM$, let us define the following spaces
\[
A(x,v):=\{w\in T_xM: K(v,w)=-1\}.
\]
\[
B(x,v):=\left\{w\in T_xM: K(v,w)=-\frac{1}{4}\right\}.
\]
 We point out that Kähler manifolds are examples of locally symmetric metrics, i.e., their curvature tensor is parallel ($\nabla R\equiv 0$), and for the locally symmetric manifolds (cf. \cite{JOST08}) we get that
\[
\{v\}^\perp = A(x,v)\oplus B(x,v).
\]
Denote by $Pr_C$ the projection to the space $C(x,v)$, with $C=A,B$. For any vector field $X$ we will use the following notations: if $\alpha$ is a geodesic and $\dfrac{D}{dt}=\nabla_{\alpha'}$ is the covariant derivative along $\alpha$, then
    \[
    \begin{cases}
        X_C:=&Pr_CX\\
        X_C':=&Pr_C\left(\dfrac{D X}{dt}\right)\\
        (X_C)':=&\dfrac{D}{dt}(Pr_C X)\\
        X_{C'}:=&\left(\dfrac{D}{dt}Pr_C\right)X\\
    \end{cases}
    \]
The above locally symmetric manifolds satisfy the property that the subbundles $C(x,v)$ are parallel, i.e., the parallel transport of vectors in $C(x,v)$ along a smooth curve $\alpha$ remains in the corresponding space $C(\alpha,\alpha')$. Then for any vector field $X$ along a geodesic $\alpha$ we have
    \[
    \dfrac{D}{dt}(Pr_C X)=Pr_C\left(\dfrac{D X}{dt}\right). 
    \]
Another way to state this is the following: saying that $C$ is parallel also means that the covariant derivative of its projection is parallel. Remember that for any endomorphism $F: TM\rarrow TM$ (also $(1,1)$-tensor) its covariant derivative is given by
    \[
    \nabla_X F= \nabla_X\circ F-F\circ \nabla_X.
    \]
    So, if $Pr_C$ is parallel along every geodesic $\gamma$, then $$(\nabla_{\gamma'}Pr_C)=\nabla_{\gamma'}\circ Pr_C-Pr_C\circ \nabla_{\gamma'}=0.$$
    Therefore the above notation gives us $X_{C'}=0$ and $(X_C)'=X_C'$.
In addition, the parallel transport of the spaces $A$ and $B$ along a closed prime geodesic preserves orientation. We are also going to make use of the following relation that holds for the above metrics (cf. \cite{JOST08}):
\begin{equation}
    \label{curvature locally symmetric}
R(X,v)v=-\frac{1}{4}X_{B(x,v)}-X_{A(x,v)}.
\end{equation}

\subsection{Geodesic flow}
For a given $\theta=(p,v) \in TM$, {we define} $\gamma_{_{\theta}}(t)$ {as} the unique geodesic with initial conditions $\gamma_{_{\theta}}(0)=p$ and 
$\gamma_{_{\theta}}'(0)=v$. For $t\in \mathbb{R}$, the family of diffeomorphisms  $g_t:TM \to TM$ define by 

\[g_{t}(\theta)=(\gamma_{_{\theta}}(t),\gamma_{_{\theta}}'(t)).
\]
This is called the \textit{geodesic flow}.\\
To study the dynamical behavior of $g_t$, we need to introduce special coordinates in $T(TM)$ (tangent bundle of $TM$).  Let $V:=\operatorname{ker}\, d\pi$ be the \textit{vertical subbundle} of $T(TM)$, where $\pi\colon TM \to M$ is the canonical projection. Let $\mathcal{K}\colon T(TM)\to TM$ be the Levi-Civita connection map of $M$ and $H:=\operatorname{ker} \mathcal{K}$ be \textit{the horizontal subbundle}. It is easy to see that $T_{_{\theta}}TM = H(\theta) \oplus V(\theta)$. Moreover, we can identify $T_{_{\theta}}TM$ and 
$T_pM \times T_pM$ through  the linear isomorphism $j_{_{\theta}}:T_{_{\theta}}TM \rightarrow T_pM \times T_pM$ given by 
	$$ j_{_{\theta}}(\xi) = (d_{_{\theta}}\pi(\xi),\mathcal{K}_{_{\theta}}(\xi)).$$
 With these coordinates, the \textit{geodesic vector field}, i.e. the vector field $G:TM\rightarrow TTM$ that generates the flow $g_t$ is given by the simple expression
 \[
 G((p,v))=(v,0).
 \]
The \textit{Sasaki metric} is a metric that makes $H(\theta)$ and $V(\theta)$ orthogonal and is given by
\[
\wh{g}_{_{\theta}}(\xi,\eta) = g(d_{_{\theta}}\pi(\xi), d_{_{\theta}}\pi(\eta) )+g(\mathcal{K}_{_{\theta}}(\xi),    \mathcal{K}_{_{\theta}}(\eta)  ).
\]
Observe that the unit tangent bundle $T^1M=\{v\in TM: g(v,v)=1\}$ is a compact manifold that is invariant under $g_t$, thus,  from now on, we consider $g_t$ restricted to $T^1M$ and $T^1M$ endowed with the  Sasaki metric.\\
\ \\
The Jacobi fields have an important role in the study of the dynamics of the geodesic flow, in fact, they are important geometrical tools to understand the behavior of the differential of the geodesic flow. A vector field $J$ along $\gamma_{_{\theta}}$ is called a \emph{Jacobi field} if it satisfies the equation 
\begin{equation}\label{Jacobi}
J'' + R(J,\gamma'_{_{\theta}})\gamma'_{_{\theta}} = 0,
\end{equation}
where $R$ is the Riemann curvature tensor of $M$ and $``\,'\,"$ denotes the covariant derivative along $\gamma_{_{\theta}}$. \\
For  $\theta=(p,v)$ and  $\xi=(w_{1},w_2) \in T_{_{\theta}}T^1M$, (the horizontal and vertical decomposition) with $w_1, w_2 \in T_{p} M$ and $\langle v, w_2 \rangle =0$. It is known that
\begin{equation}\label{eq4}
(dg_{t})_{\theta}(\xi) = (J_{_{\xi}}(t), J'_{\xi}(t)),
\end{equation}  
where $J_{_{\xi}}$ denotes the unique Jacobi vector field along $\gamma_{_{\theta}}$ such that $J_{_{\xi}}(0) = w_1$ and $J'_{\xi}(0) = w_2$ (see \cite{Paternain}). The equation (\ref{eq4}) enables us to assert that the investigation of the dynamics of the geodesic flow revolves around Jacobi fields. \\
The Riemannian metric $g$ induces a symplectic $2$-form on $TM$ and a contact form on $T^1M$, which are invariant under the geodesic flow, in the following way: given $(x_1,y_1),(x_2,y_2)\in T_\theta TM$ we can define the symplectic form by
\[
\omega_\theta((x_1,y_1),(x_2,y_2))=g(x_1,y_2)-g(y_1,x_2),
\]
and for $\xi\in T_\theta TM$, with $\theta=(p,v)$ we define the $1$-form:
\[
\alpha_\theta(\xi)=\wh{g}(\xi,G(\theta))=g(d_\theta \pi(\xi),v).
\]
When $\alpha$ is restricted to $T^1M$, it becomes a contact form with \textit{Reeb vector field} $G$. Thus, the volume form $\alpha\wedge (d\alpha)^{2n-1}$ induces a probability measure on $T^1M$ called Liouville measure. For $\theta\in T^1M$ we will write $S(\theta)=\ker \alpha_\theta$, which is the set of orthogonal vectors to $G(\theta)$ with respect to the Sasaki metric. The subbundle $S(T^1M)\subset TT^1M$ is called \textit{Contact structure of the unit tangent bundle}.

\noindent One of the geodesic flows that we discuss in this paper is the Anosov geodesic flow, whose definition follows below.\\
We say that the geodesic flow $g_t:T^1M \rightarrow T^1M$ is Anosov (with {respect to} the Sasaki metric on $T^1M$) if $T(T^1M)$ has a splitting $T(T^1M) = E^{ss} \oplus \langle G \rangle \oplus E^{uu} $ such that 
\begin{flalign*}
&\begin{aligned}
\, \, \,\,\, \hspace{0.5cm} &(dg_t)_{\theta} (E^{ss}(\theta)) =E^{ss}(g_t(\theta)),\\
\, \, \,\,\, &(dg_t)_{\theta} (E^{uu}(\theta)) = E^{uu}(g_t(\theta)),\\
\, \, \,\,\, & ||(dg_t)_{\theta}\big{|}_{E^{ss}}|| \leq C e^{\lambda t},\\
\, \, \,\,\,& ||(dg_{-t})_{\theta}\big{|}_{E^{uu}}|| \leq C e^{\lambda t},\\
\end{aligned}&&
\end{flalign*}

\noindent for all $t\geq 0$ with $ C > 0$ and $\lambda <0$,  where $G$ is the vector field derivative of the geodesic flow.\\
The geodesic flow of a compact manifold with negative curvature is Anosov (cf. \cite{Anosov}). There are also examples of compact manifolds with regions of zero and positive curvature whose geodesic flow is Anosov (\cite{Donnay} and \cite{EBERLEIN73}). Nevertheless, the Anosov condition is inherently connected to the presence of negative curvature.
\begin{proposition}\emph{(Corollary 3.4 in \cite{EBERLEIN73})}
\label{Eberleins criterion}
    Let $(M,g)$  be a Riemannian manifold. If the geodesic flow $g_t$ is of Anosov type, the following holds: Let $\alpha$ be any unit speed geodesic of $M$, and $X(t)$ be any nonzero perpendicular parallel vector field along $\alpha$. Then the sectional curvature $K(X, \alpha')(t) < 0$ for some $t\in \R$.
\end{proposition}

The last proposition will be useful since we need to nullify the sectional curvature along some geodesics to disrupt the Anosov condition.

\subsection{Partial hyperbolicity}
In this section, we are going to present all the definitions in the particular setting of the geodesic flow, however, the definitions are the same for the general case. The classical definition of partial hyperbolicity is the following
\begin{definition}
 A geodesic flow $g_t\colon T^1M \to T^1M$ is \emph{partially hyperbolic} if there exists a nontrivial $dg_t$-invariant  splitting $T(T^1M)/ \langle G \rangle = E^{ss} \oplus E^c \oplus E^{uu}$  such that
\begin{flalign*}
&\begin{aligned}
\, \, \,\,\, \hspace{0.5cm} &\|(dg_{t})_{\theta}|_{E^{ss}}\|<C e^{\lambda t},\\
\, \, \,\,\, & \|(dg_t)_{\theta}|_{E^{uu}}\|\leq C e^{-\lambda t},\\
\, \, \,\,\, & C e^{\mu t}\leq  \|(dg_t)_{\theta}|_{E^c}\|\leq C e^{-\mu t},\\
\end{aligned}&&
\end{flalign*}   
for some $\lambda<\mu\leq 0<C$ and for all $\theta \in T^1M$.
\end{definition}
\noindent Another way to present some kind of partial hyperbolicity appears as the so-called \textit{dominated splitting} which we present below
\begin{definition}
 A $dg_t$-invariant splitting $E\oplus F$ of $T(T^1M)/\langle G \rangle$ is called a dominated splitting if 
 $$\|(dg_t)_{\theta}|_{E(\theta)}\| \cdot \|(dg_{-t})_{g_t(\theta)}|_{F(g_t(\theta))}\|<C e^{-\lambda t},$$
 for some constants $C,\lambda>0$ and $\theta \in T^1M$.
\end{definition}
Dominated splitting and the uniform contraction expansion in $E^{ss}$ and $E^{uu}$ in the definition of partial hyperbolicity are in general different. However in the special case of symplectic dynamics (as the case of geodesic flows), we have a beautiful connection \cite{CONTRERAS2002partially} and \cite{RUGGIERO1991persistently} (see also \cite[Lemma 2.8] {CARNEIRO&PUJALS14}). 
\begin{lemma}
\label{dominated splitting in symplectic}
 Let $(N,\omega)$ be a symplectic manifold, $\omega$  its symplectic $2$-form and $\phi^t \colon N \to N$ a flow on $N$ such that
 $\mathcal{L}_{X}(\omega)=0$, \emph{i.e.}, it preserves the symplectic structure of $N$. If there is a dominated splitting $TN/X =E\oplus E^c \oplus F$
 such that $dim(E) = dim(F)$, then for all $x \in N$ there are a positive real number $C$ and a negative real number $\lambda$ such
 that
 $$\|d\phi_x^t|_E\| \leq Ce^{\lambda t }, \,\,\,\,\,\,\,\, \|d\phi_x^{-t}|_F\| \leq Ce^{\lambda t }.$$
\end{lemma}

\subsubsection{Cone invariance criteria for partial hyperbolicity}
\label{cone criteria}
This section is similar to \cite[Section 2.2.1]{CARNEIRO&PUJALS14}, but we include it for the completeness of our arguments. Classical results on hyperbolicity and partial hyperbolicity utilize the \emph{cone criteria}, a method used to produce invariant splittings.
Given $\theta \in T^1M$, a subspace $E(\theta)\subset T_{\theta}T^1M$, and a positive number $\delta$, we define the cone at $\theta$ centered around $E(\theta)$ with angle $\delta$ (or opening of the cone) as 
$$C(\theta, E(\theta), \delta):=\{\xi\in T_{\theta}T^1M: \angle (\xi, E(\theta))< \delta\},$$
where $\angle (\xi, E(\theta))$ is the angle between $\xi$ and $E(\theta)$. 
The geodesic flow $g_t$ is partially hyperbolic if there are $\delta>0$, $T>0$, and two continuous families of cones $C(\theta, E_1(\theta), \delta)$  and $C(\theta, E_2(\theta), \delta)$ such that: 
\begin{flalign*}
&\begin{aligned}
\, \, \,\,\, \hspace{0.5cm} &(dg_t)_{\theta} (C(\theta, E_2(\theta), \delta)) \varsubsetneqq C(\theta, E_2(g_t(\theta)), \delta),\\
\, \, \,\,\, &(dg_{-t})_{\theta} (C(\theta, E_1(\theta), \delta)) \varsubsetneqq C(\theta, E_1(g_{-t}(\theta)), \delta),\\
\, \, \,\,\, & ||(dg_t)_{\theta}(\xi_1)|| <  Ke^{\lambda t},\\
\, \, \,\,\,& ||(dg_{-t})_{\theta}(\xi_2)|| <  Ke^{\lambda t},\\
\end{aligned}&&
\end{flalign*}
for all $t>0$, $\xi_1 \in C(\theta, E_1(\theta), \delta)$, $\xi_2 \in C(\theta, E_2(\theta), \delta)$, and some constants $K>0$, $\lambda<0$.\\
In the following, we present a strategy to prove that a family of cones is invariant. Let $E$
be a vector bundle over $T^1M$ which is a subbundle of $TT^1M$ and $\pi_{E}\colon E \to T^1M$ its canonical projection. Let $\text{Pr}_{E} \colon TT^1M \to E$ be the orthogonal projection to $E$. We define the real function $\Theta_{E}\colon TT^1M \to \mathbb{R}$ as 
\begin{equation}
    \Theta_{E}(\xi):= \dfrac{g(\text{Pr}_{E}(\xi), \text{Pr}_{E}(\xi))}{g(\xi,\xi)}.
\end{equation}

\begin{lemma}\emph(\emph{\cite[Lemma 2.10]{CARNEIRO&PUJALS14}})
For a $\delta>0$, and a fixed vector bundle $E$ on $T^1M$, if 
\begin{equation}
    \dfrac{d}{dt}\Theta_{E(g_t(\theta))}((dg_t)_\theta(\xi))>0,
\end{equation}
for $\xi \in \partial C(\theta, E(\theta), \delta):=\{\xi\in T_{\theta}T^1M: \angle (\xi, E(\theta))=\delta\}$, then the family of cones $C(\theta, E(\theta), \delta)$ is invariant for the geodesic flow.  
\end{lemma}

\section{Breaking the Anosov property}
\label{Breaking the Anosov Property}
The goal of our construction is to produce interesting dynamics for the geodesic flow outside the class of Anosov flows. The first step is to carry out a deformation of the initial Anosov metric in a controlled way by adjusting the effect on specific sectional curvatures. Remember that we start with a locally symmetric Riemannian manifold of varying negative curvature, which is, in particular, $C^2$-stably Anosov. By looking at this situation through the lens of Proposition \ref{Eberleins criterion} and considering the $C^2$-stability of Anosov metrics, one realizes that it is necessary to make a large $C^2$-deformation on the metric, and a sufficient condition to destroy the Anosov property is the creation of a geodesic $\alpha$ that admits a perpendicular parallel vector field $X$ with the property that the planes $\Pi(t)=\text{span}\{\alpha'(t),X(t)\}$ have zero sectional curvature for every time $t\in \R$.

The technique we are going to use is given by conformal transformations of the initial metric. This type of change of the Riemannian metric has been used in different contexts and it allows us to have a good control on how the geometric quantities are varying with the deformation (see Subsection \ref{preliminaries riemannian geometry}). In this section, we are going to explore this type of deformation within our setting and derive naturally conditions that lead to the breaking of the Anosov property.

We are going to perform a deformation of the metric $g$ in a tubular neighborhood of a closed prime geodesic. By doing so we can localize the deformation and have better control over the dynamics of the geodesic flow. To produce zero curvature, one could try to consider non-closed geodesics, but then the effect could not be local, and the analysis of the dynamics can be way more delicate. It is also worth mentioning that a closed prime geodesic represents a periodic orbit for the geodesic flow.

Let $\gamma$ be a closed geodesic with period $T$ and $\gamma'(0)=\gamma'(T)$ just as in \cite{CARNEIRO&PUJALS14}. For each $x \in M$ and $v\in T_xM$ remember that we have defined
\[
A(x,v):=\{w\in T_xM: K(v,w)=-1\}.
\]
\[
B(x,v):=\left\{w\in T_xM: K(v,w)=-\frac{1}{4}\right\}.
\]
We consider an orthonormal parallel frame $\{e_0(t):=\gamma'(t),...,e_{n-1}(t)\}$ of $T_{\gamma(t)}M$ such that $\{e_1(t),...,e_{r}(t)\}$ is a basis for $A(\gamma(t),\gamma'(t))$ and $\{e_{r+1}(t),...,e_{n-1}(t)\}$ is a basis for $B(\gamma(t),\gamma'(t))$. This can be done by choosing such a basis for $T_{\gamma(0)}M$ and then considering its parallel transport, since the metric is locally symmetric, the respective vectors are still in the respective spaces $A$ and $B$. Define the Fermi coordinates $\Psi: [0,T]\times (-\ve_0,\ve_0)^{n-1}\rarrow M$ by
$$\Psi(t,x)=\mrm{exp}_{\gamma(t)}(x_1e_1(t)+...+x_{n-1}e_{n-1}(t)),$$
with $\ve_0$ less than the injective radius. In particular, for Fermi coordinates we have that $g_{ij}(t,0)=\delta_{ij}$ and $\Gamma_{ij}^k(t,0)=0$. For each $\ve<\ve_0$, we define the following sets

\begin{itemize}
    \item $U:=[0,T]\times (-\ve_0,\ve_0)^{n-1}$.
    \item $U(\ve):=[0,T]\times (-\ve,\ve)^{n-1}$.
    \item $B(\gamma,\ve)=\Psi(U(\ve))$.
\end{itemize}
The choice of $\ve$ small enough will be important to control the deformation of the metric up to the first derivatives of its component functions. 

We consider a conformal deformation of the initial metric given by $\til{g}=\phi g$, with $\phi=\mrm{e}^h$ and $h$ supported in $U(\ve)$ to be determined. We will construct the function $h$ to maintain $\gamma$ as a geodesic for the new metric, and to ensure that $\til{g}$ is $C^1$-close to $g$ and $C^2$-far. It means that we will make a small perturbation on the norms and Christoffel symbols of $g$, but a deformation on the curvature. Notice that if $h$ is $C^1$ small, then $\til{g}$ is $C^1$ close to $g$. Besides that, all pairs of orthogonal vectors for the metric $g$ are still orthogonal for the metric $\til{g}$. Furthermore, if we set $h(t,0)=0$, then $\{e_0(t),....,e_{n-1}(t)\}$ is still an orthonormal basis for $T_{\gamma(t)}M$ considering the metric $\tilde{g}$.

The first property we need to obtain for the new metric is that $\gamma$ is still a $\til{g}$-geodesic. For any curve $\alpha$ we will denote by $\til{D}_t = \til{\nabla}_{\alpha'}(\cdot)$ and $D_t=\nabla_{\alpha'}(\cdot)$ the covariant derivative along  $\alpha$ given by the Levi-Civita connection of the metrics $\til{g}$ and $g$, respectively. By the equation \eqref{conformal connection} we get for $\gamma$
\begin{equation}
\label{relation geodesics}
\til{D}_t\gamma'=D_t\gamma'+\gamma'(h)\gamma'-\nabla h|_{\gamma}=g(\nabla h,\gamma')\gamma'-\nabla h|_{\gamma}.
\end{equation}
It is sufficient  to have that $\gamma$ is a critical set for $h$, i.e. $\nabla h|_{\gamma'}=0$. Furthermore, in this case, by the same calculation, every parallel vector field for the metric $g$ along $\gamma$ is still parallel for the metric $\til{g}$. In particular for every $k=1,...,n-1$ the vector field $e_k(t)$ is parallel along $\gamma$ for the metric $\til{g}$.

We now investigate the new curvature in Fermi coordinates. We are interested in $\til{K}\left(\dfrac{\partial}{\partial x^k},\gamma'\right)$ given by  the equation \eqref{conformal sectional curvature}. In coordinates we have
\[
\nabla h=g^{ij}\dfrac{\partial h}{\partial x^j}\dfrac{\partial}{\partial x^i},
\]
and
\begin{align*}
    \nabla_X(\nabla h)&=X\left(g^{ij}\dfrac{\partial h}{\partial x^j}\right)\dfrac{\partial}{\partial x^i}+g^{ij}\dfrac{\partial h}{\partial x^j}\nabla_X\left(\dfrac{\partial}{\partial x^i}\right)\\
                    &=\left(X(g^{ij})\dfrac{\partial h}{\partial x^j}+g^{ij}X\left(\dfrac{\partial h}{\partial x^j}\right)\right)\dfrac{\partial}{\partial x^i}+g^{ij}\dfrac{\partial h}{\partial x^j}\nabla_X\left(\dfrac{\partial}{\partial x^i}\right).
\end{align*}
If $X_k:=\dfrac{\partial}{\partial x^k}$, then
\begin{align*}
    \nabla_{X_k}(\nabla h)&=\left(X_k(g^{ij})\dfrac{\partial h}{\partial x^j}+g^{ij}X_k\left(\dfrac{\partial h}{\partial x^j}\right)\right)\dfrac{\partial}{\partial x^i}+g^{ij}\dfrac{\partial h}{\partial x^j}\nabla_{X_k}X_i\\
    &=\left(X_k(g^{ij})\dfrac{\partial h}{\partial x^j}+g^{ij}\dfrac{\partial^2 h}{\partial x^k\partial x^j}\right)X_i+g^{ij}\dfrac{\partial h}{\partial x^j}\Gamma_{ki}^lX_l\\
    &=\left(X_k(g^{ij})\dfrac{\partial h}{\partial x^j}+g^{ij}\dfrac{\partial^2 h}{\partial x^k\partial x^j}\right)X_i+g^{lj}\dfrac{\partial h}{\partial x^j}\Gamma_{kl}^iX_i\\
    &=\left(X_k(g^{ij})\dfrac{\partial h}{\partial x^j}+g^{ij}\dfrac{\partial^2 h}{\partial x^k\partial x^j}+g^{lj}\dfrac{\partial h}{\partial x^j}\Gamma_{kl}^i\right)X_i.
\end{align*}
Now, using Lemma \ref{derivative of product metric by  inverse} we get
\begin{align*}
g(\nabla_{X_k}\nabla h,X_k)&=\left(X_k(g^{ij})\dfrac{\partial h}{\partial x^j}+g^{ij}\dfrac{\partial^2 h}{\partial x^k\partial x^j}+g^{lj}\dfrac{\partial h}{\partial x^j}\Gamma_{kl}^i\right)g_{ik}\\
&=X_k(g^{ij})g_{ik}\dfrac{\partial h}{\partial x^j}+g^{ij}g_{ik}\dfrac{\partial^2 h}{\partial x^k\partial x^j}+g^{lj}g_{ik}\dfrac{\partial h}{\partial x^j}\Gamma_{kl}^i\\
&=\partial_k(g^{ij})g_{ik}\partial_j h+\partial_{kk}^2 h+g^{lj}g_{ik}\partial_j h\Gamma_{kl}^i\\
&=-g^{ij}\partial_kg_{ik}\partial_j h+\partial_{kk}^2 h+g^{lj}g_{ik}\partial_j h\Gamma_{kl}^i\\
&=-g^{ij}(\Gamma_{ki}^lg_{lk}+\Gamma_{kk}^lg_{li})\partial_j h+\partial_{kk}^2 h+g^{lj}g_{ik}\partial_j h\Gamma_{kl}^i\\
&=-g^{ij}g_{li}\Gamma_{kk}^l\partial_jh+\partial_{kk}^2h\\
&=-\delta_{l}^j\Gamma_{kk}^l\partial_jh+\partial_{kk}^2h\\
&=-\Gamma_{kk}^j\partial_jh+\partial_{kk}^2h.
\end{align*}
In particular,
\[
g(D_t\nabla h,\gamma')=\partial_{00}^2h-\Gamma_{00}^k\partial_kh=\partial_{00}^2h(t,x).
\]
Then the new sectional curvatures along $\gamma$ are given by
\[
\phi\til{K}(X_k,\gamma')=K(X_k,\gamma')-\frac{1}{2}\partial_{kk}^2h(t,0)-\frac{1}{2}\partial_{00}^2h(t,0).
\]
If $h$ does not depend on $t$, we get

\begin{enumerate}
    \item $\phi\til{K}(X_k,\gamma')=-1-\frac{1}{2}\partial_{kk}^2h(0)$, for $k=1,...,r$.
    \item $\phi\til{K}(X_k,\gamma')=-\frac{1}{4}-\frac{1}{2}\partial_{kk}^2h(0)$, for $k=r+1,...,n-1$.
\end{enumerate}
\begin{remark}
    Given a geodesic $\alpha$ and a parallel orthonormal vector field $X(t)$ along $\alpha(t)$, we have that $X(t)$ is a Jacobi field if and only if $K(X,\alpha')(t)\equiv 0$.
\end{remark}
By the above remark, if we choose the function  $h$ such that for some $s\in \{r+1,...,n-1\}$ we have $-\frac{1}{2}\partial_{ss}^2h(0)=\frac{1}{4}$, then $e_s(t)$ is a Jacobi field for the metric $\til{g}$ along $\gamma'$. From Proposition \ref{Eberleins criterion} we conclude that the geodesic flow for $\til{g}$ is not Anosov. Concretely, we have
\[
d\til{g}_t(e_s(0),0)=(e_s(t),0),
\]
and so, for every $t\in \R$,
\[
\norm{d\til{g}_t(e_s(0),0)}=\norm{e_s(t)}=\norm{e_s(0)}.
\]
The last equality is due to the fact that $e_s(t)$ is parallel. It means that the derivative of the geodesic flow $\til{g}_t$ acts as an isometry on the vector $(e_s(t),0)$. In summary, to break the Anosov property of the metric $g$ it  is enough that $h$ satisfies the following conditions
\begin{enumerate}
    \item $h$ does not depend on $t$.
    \item $h(0)=0$ and $\nabla h(0)=0$.
    \item There exist $s\in \{r+1,...,n-1\}$ such that $-\frac{1}{2}\partial_{ss}^2h(0)=\frac{1}{4}$.
\end{enumerate}
To avoid the creation of planes with positive curvature along $\gamma$ we are going to choose $h$ such that the property $3.$ above is satisfied by a unique index $s$ and such that $\partial_{kk}^2h(0)=0$ for all other indices. By doing so we guarantee that $\gamma'$ is a partially hyperbolic orbit for $\til{g}_t$ as described below.

We now construct the function $h$ with the above properties. For $n\geq  2$, let $s_{n}$ be the following functions $s_n:\R\rarrow \R$
\[
s_n(x)=
\begin{cases}
1/8(x+1)^{2n}(x-1)^{2n},&x\in[-1,1]\\
0, &\mbox{ otherwise }
\end{cases}.
\]
Define also the functions $r_n(x)=x^2s_n(x)$. The family $\{r_n\}_n$ is smooth in $\R\setminus\{\pm1\}$ and is of class $C^{2n}$ in $\{\pm1\}$. It is easy to see that
\begin{itemize}
    \item $r_n(0)=0$.
    \item $r_n'(0)=0$.
    \item $r_n''(0)=\frac{1}{4}$.
\end{itemize}
It is also not difficult to see that $|r_n''(x)|<\frac{1}{4}$ if $x\neq 0$. Now, let $f_n:\R\rarrow\R$ be the function $f_n(x)= x^2 s_n\left(\frac{x}{\ve^2}\right)$ which is smooth in $\R\setminus\{\pm \ve^2\}$ and of class $C^{2n+2}$ in $\{\pm \ve^2\}$. The family $\{f_n\}_n$ satisfies the same properties as $r_n$, we just changed the support of it to be $[-\ve^2,\ve^2]$ and we also have $f_n(x)\leq \ve^4$. The graphs of $r_n, r_n'$, and $r_n''$ for $n=2$ are shown below in Figure 1.
\begin{figure}[!htb]
 \centering
 \label{fig:images}
 \subfloat[Graph of $r_2$]{%
      \includegraphics[width=0.29\textwidth]{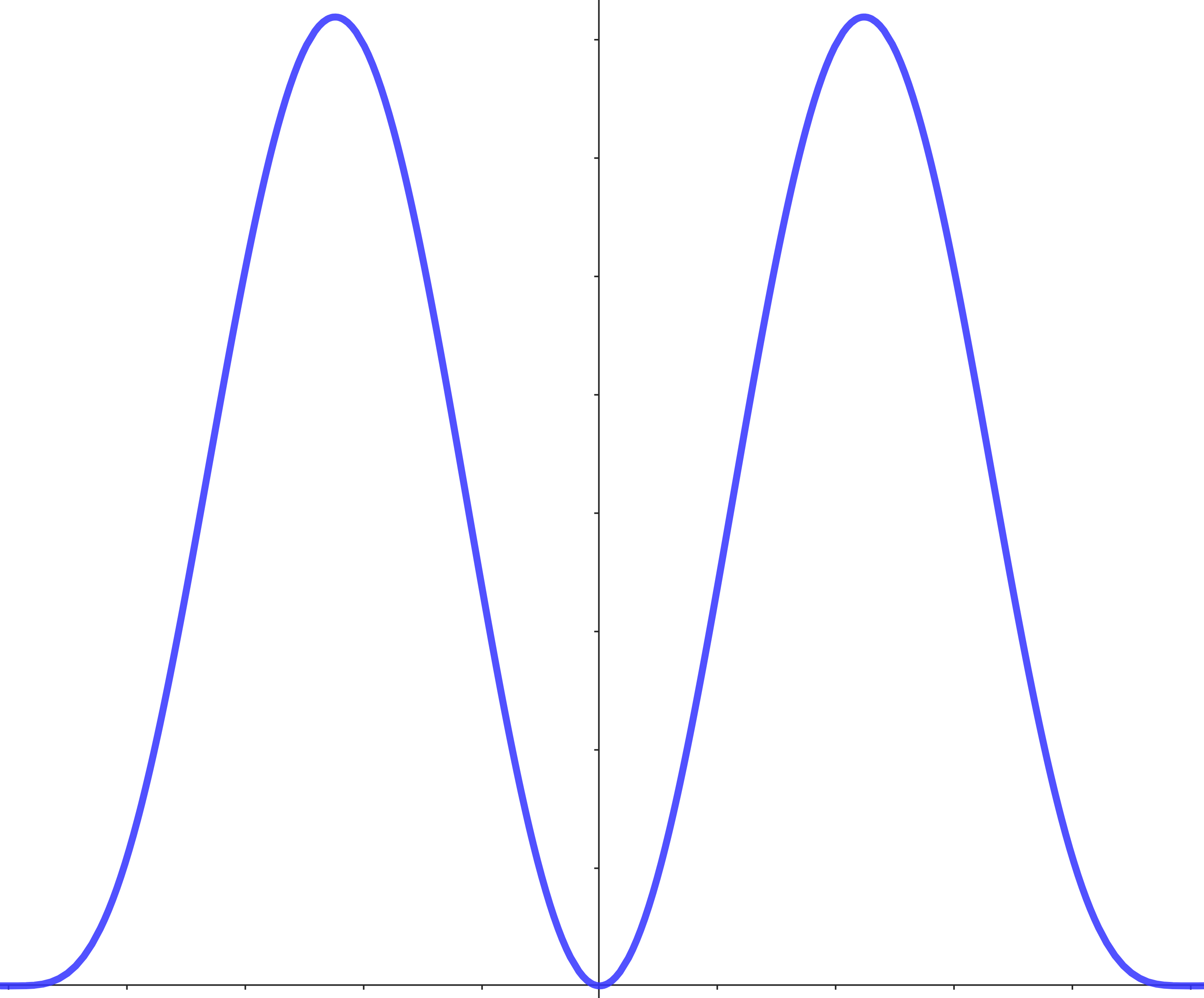}}
      \label{fig:image-a}
 \qquad
 \subfloat[Graph  of  $r'_2$]{%
      \includegraphics[width=0.28\textwidth]{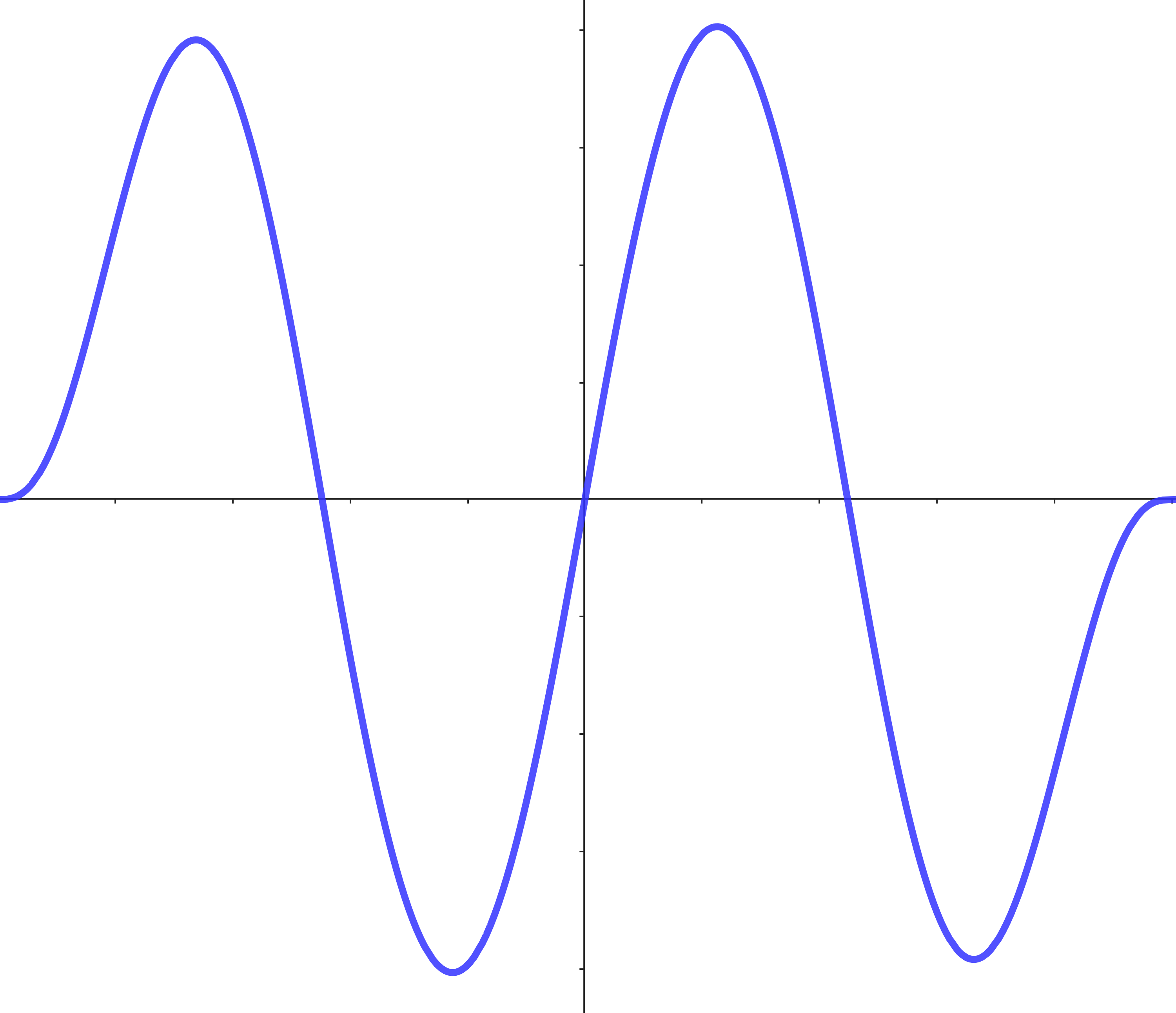}}
      \label{fig:image-b}
 \qquad
 \subfloat[Graph  of  $r''_2$]{%
      \includegraphics[width=0.28\textwidth]{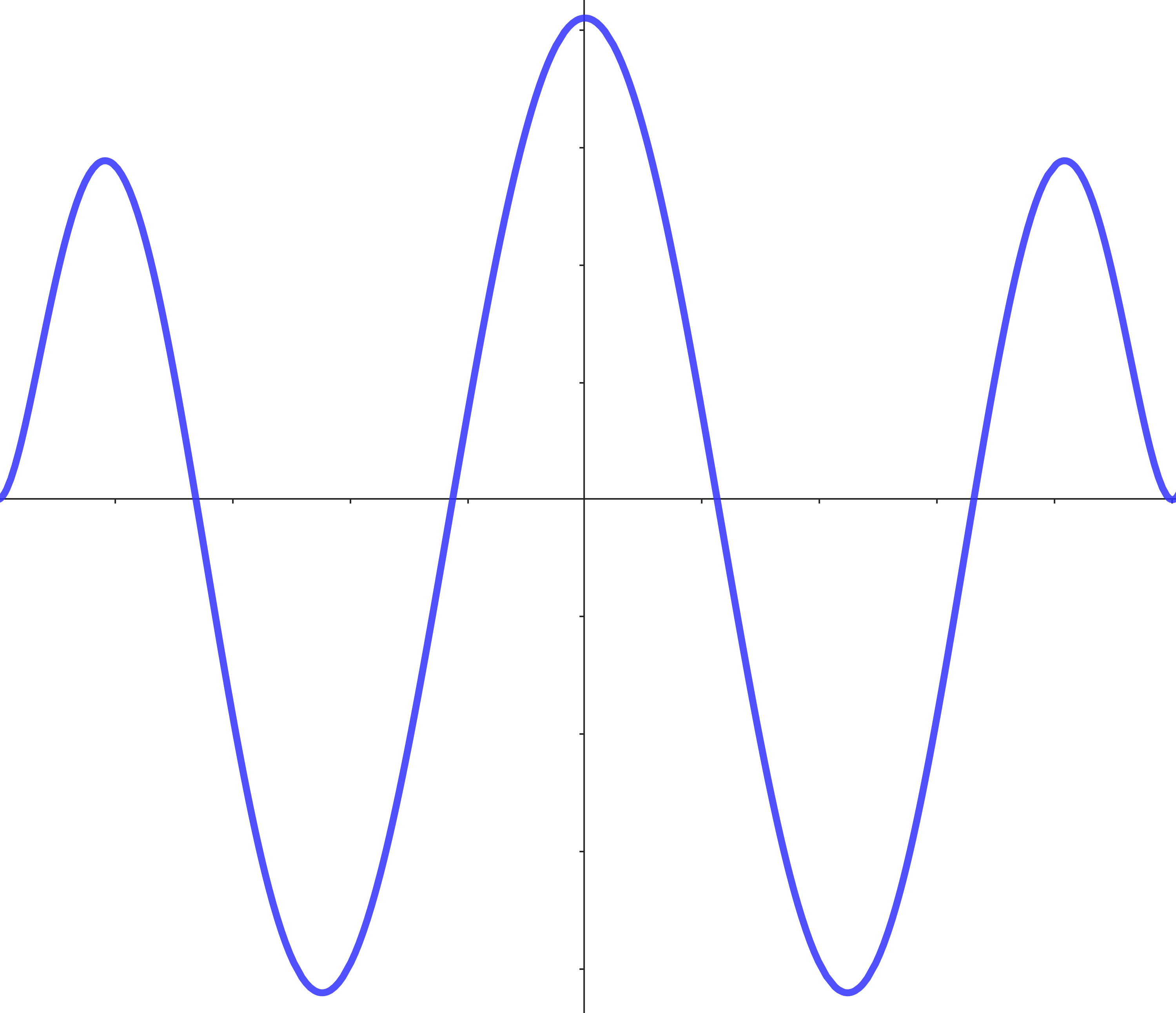}}
      \label{fig:image-c}
      \caption{Behavior of $r_2$}%
\end{figure}

\noindent Finally, for a fixed $s\in \{r+1,...,n-1\}$ we define $h_k$ as the product
\[
h_k(x_0,x_1,...,x_{n-1}):=\phi\left(\frac{x_1}{\ve}\right)\cdots \left(-2f_k(x_s)\right)\cdots\phi\left(\frac{x_{n-1}}{\ve}\right),
\]
where $\phi(x)$ are bump functions  such that $\phi(0)=1$ and supported in $[-1,1]$. For simplicity, let us denote  $\Phi(x)=\phi\left(\frac{x_1}{\ve}\right)\cdots\phi\left(\frac{x_{n-1}}{\ve}\right)$. Notice  that
$h_k$ is smooth outside  the region $x_s=\pm \ve^2$, where it is of class $C^{2k+2}$.

By doing so, we get the following properties
\begin{lemma}
$h_k$ satisfies the following properties
\begin{enumerate}
    \item $h_k(0)=0$.
    \item $h_k$ does not depend on $t=x_0$.
    \item $\partial_j h_k(0)=0$, for all $j$.
    \item $\partial_{ij}^2h_k(0)=0$, if $i\neq j$ or $i=j\neq s$.
    \item $\partial_{ss}^2h_k(0)=-\frac{1}{2}$.
\end{enumerate}
\end{lemma}
\begin{remark}
    It is important to highlight that the deformation needs to maintain the central geodesic $\gamma$ as a unit-speed geodesic for the new metrics $\tilde{g}_k=\mrm{e}^{h_k}g$. According to the equation \eqref{relation geodesics}, this is given by properties $1.$, $2.$ and $3.$ from the lemma above.
\end{remark}
Notice that
\[
    |-2f_k(x)|\leq 2\ve^4\left|r_k\left(\frac{x}{\ve^2}\right)\right|\leq \frac{2}{100}\ve^4<\ve^4,
\]
\begin{align*}
   |-2f_k'(x)|&=2\left|2xr_k\left(\frac{x}{\ve^2}\right)+x^2 r_k'\left(\frac{x}{\ve^2}\right)\frac{1}{\ve^2}\right|\\
   &\leq 4\ve^2\frac{1}{100}+\ve^2\frac{4}{100}< \ve^2.
\end{align*}
We also have $(\phi\left(\frac{x}{\ve}\right))'\leq\frac{2}{\ve}$ and $(\phi\left(\frac{x}{\ve}\right))''\leq \frac{M}{\ve^2}$. Once $|\phi(x)|\leq 1$, it follows that

\begin{lemma}
\label{estimates on h}
The following inequalities hold:
    \begin{enumerate}
    \item $|h_k(x)|\leq \ve^4$.
    \item $|\partial_j h_k|\leq \frac{2}{\ve}\ve^4=2\ve^3$, if $j\neq s$.
    \item $|\partial_s h_k|\leq \ve^2$.
    \item $|\partial_{ij}^2h_k|\leq \frac{2}{\ve}\frac{2}{\ve}\ve^4=4\ve^2$, if $i\neq j$ and non of them are equal to $s$.
    \item $|\partial_{sj}^2h_k|\leq \ve^2\frac{2}{\ve}=2\ve$, if $i\neq s$.
    \item $|\partial_{ii}^2h_k|\leq \frac{M}{\ve^2}\ve^4=M\ve^2$, if $i\neq s$.
    \item $|\partial_{ss}^2h_k|\leq 2|f_k''(x_s)|\leq2\left|r_k''\left(\frac{x_s}{\ve^2}\right)\right| \leq\frac{1}{2}$.
\end{enumerate}
\end{lemma}

Now fix some $k_0\geq 2$ and call $h=h_{k_0}$.

\section{The new geodesic flow is partially hyperbolic}
\label{The new geodesic flow is Partially Hyperbolic}
We start this section by illustrating the behavior of the new flow. In the next Lemma, we see that the orbit $\{\til{g}_t(\gamma'(0))\}$ is not hyperbolic, but it is still partially hyperbolic. We can also verify precisely what the invariant splitting for the deformed flow is. It will lead to the analysis of other trajectories $\til{g}_t$.

\begin{lemma}
\label{gamma is partially hyperbolic}
    The orbit $\{\til{g}_t(\gamma'(0))\}$ is partially hyperbolic.
\end{lemma}
\begin{proof}
    We will analyze the Jacobi equation for the new metric along $\gamma$. Define by $\mc{J}(\gamma)$ the space of Jacobi fields orthogonal to $\gamma'$. Since $\{e_1(t),...,e_{n-1}(t)\}$ is an orthonormal parallel frame along $\gamma$ we can write $J\in\mc{J}(\gamma)$ as $J=J^i(t)e_i(t)$, then denoting $\til{D}_t^2J=J''$ we get
    \[
J''=(J^i)''e_i=-J^i\til{R}(e_i,\gamma')\gamma'.
    \]
    Then, we get
    \[
    \til{g}(J'',e_j)=(J^i)''\til{g}_{ij}=-J^i\til{K}_{ij},
    \]
    with $\til{K}_{ij}=\til{g}(\til{R}(e_i,\gamma')\gamma',e_j)=\til{R}(e_i,\gamma',\gamma',e_j)$. From the computations above, we have the following relation
    \[
\tilde{R}(e_i,\gamma',\gamma',e_j) = R(e_i,\gamma',\gamma',e_j)-\frac{1}{2}\partial_{ij}^2 h.
    \]
    We get
    \begin{enumerate}
        \item $\til{R}(e_i,\gamma',\gamma',e_j)=R(e_i,\gamma',\gamma',e_j)$, for $i\neq j$ and $i=j\neq s$.
        \item $\til{R}(e_s,\gamma',\gamma',e_s)=R(e_s,\gamma',\gamma',e_s)+\frac{1}{4}=0$.
    \end{enumerate}
    
 Since the metric $g$ is locally symmetric of nonconstant negative curvature, we can apply the relation \eqref{curvature locally symmetric}. Notice that, $e_i=(e_i)_A$ and $(e_i)_B=0$ if $i=1,...,r$ and if $i=r+1,...,n-1$, then  $e_i=(e_i)_B$ and $(e_i)_A=0$ . Hence if $i\neq j$
\[
R(e_i,\gamma',\gamma',e_j) = \lambda g(e_i,e_j)=0,
\]
where $\lambda\in\left\{\frac{1}{4},1\right\}$ depending on $i$ and $j$. For $i=j\in\{1,...,r\}$
\[
\tilde{R}(e_i,\gamma',\gamma',e_i) = R(e_i,\gamma',\gamma',e_i)=-1.
\]
For $i=j\in \{r+1,...,n-1\}\setminus\{s\}$
\[
\tilde{R}(e_i,\gamma',\gamma',e_i) = R(e_i,\gamma',\gamma',e_i)=-\frac{1}{4}.
\]
We conclude that the matrix $\tilde{K}=(\tilde{R}(e_i,\gamma',\gamma',e_j))_{ij}$ is

\[
\tilde{K}= \begin{pmatrix}
-Id_r  & 0 \\
0 & I_{n-r-1}
\end{pmatrix},
\]
with $I_{n-r-1}$ the diagonal matrix with entries $-\frac{1}{4}$ but the entry $(I_{n-r-1})_{s^*s^*}=0$, with $s^*=s-r$. This implies that 
\[
\begin{cases}
    (J^i)''=J^i, &\mbox{for }i=1,...,r\\
    (J^i)''=\frac{1}{4}J^i, &\mbox{for }i=r+1,...,n-1 \mbox{ and } i\neq s\\
    (J^s)''=0.
\end{cases}
\]
It implies that the ODE solutions are
\[
\begin{cases}
    J^i=\left(\frac{J^i(0)-(J^i)'(0)}{2}\right)e^{-t}+\left(\frac{J^i(0)+(J^i)'(0)}{2}\right)e^t, &\mbox{for }i=1,...,r\\
    J^i=\left(\frac{J^i(0)-2(J^i)'(0)}{2}\right)e^{-\frac{t}{2}}+\left(\frac{J^i(0)+2(J^i)'(0)}{2}\right)e^{\frac{t}{2}}, &\mbox{for }i=r+1,...,n-1 \mbox{ and } i\neq s\\
    J^s=(J^s)'(0)t+J^s(0).
\end{cases}
\]
From this, we can see that the splitting along $\gamma'$ is given by 

\begin{align*}
    E^{ss}(\gamma')=\,&\{(J,J'): J\in \mc{J}(\gamma), J^i(0)=-(J^i)'(0),  \mbox{for } i=1,...,r\\
                    &\mbox{ and } J^i(0)=(J^i)'(0)=0 \mbox{ otherwise}\},\\
    E^{uu}(\gamma')=\,&\{(J,J'): J\in \mc{J}(\gamma), J^i(0)=(J^i)'(0),  \mbox{for } i=1,...,r\\
                    &\mbox{ and } J^i(0)=(J^i)'(0)=0 \mbox{ otherwise}\},\\
    E^{s}(\gamma')\,\, =\, &\{(J,J'): J\in \mc{J}(\gamma), J^i(0)=-2(J^i)'(0),  \mbox{for } i=r+1,...,n-1, i\neq s\\
    &\mbox{ and } J^i(0)=(J^i)'(0)=0 \mbox{ otherwise}\},\\
    E^{u}(\gamma')\,\, =\,&\{(J,J'): J\in \mc{J}(\gamma), J^i(0)=2(J^i)'(0),  \mbox{for } i=r+1,...,n-1, i\neq s\\
    &\mbox{ and } J^i(0)=(J^i)'(0)=0 \mbox{ otherwise}\},\\
    E^c(\gamma)\,\,=\, \, &\mrm{span}\{(e^s(t),0)\}.
\end{align*}
\end{proof}

\subsection{Cone invariance for the initial metric}
In this section, we will analyze the behavior of other orbits of the geodesic flow $\til g_t$. For completeness, we first see that the contact structure for the initial metric splits in four invariant subbundles, as it is done in \cite{CARNEIRO&PUJALS14}. The contact structure $S(T^1M) \rarrow T^1M$ has naturally the following identification: $S(x,v)=(A(x,v)\oplus B(x,v))^2$. As mentioned before, the locally symmetric property implies that the spaces $A(x,v)$ and $B(x,v)$ are parallel, i.e. if $V_0 \in A(x,v)$ and $V(t)$ is the vector field obtained by parallel transport along a smooth curve $\gamma$, then $V(t)\in A(\gamma(t),\gamma'(t))$. Analogously for $B(x,v)$.

\begin{lemma}[\protect{\cite[Lemma~4.1]{CARNEIRO&PUJALS14}}]
\label{four invariant bundles}
    The geodesic flow of the locally symmetric spaces of non-constant negative curvature induces a hyperbolic splitting of the contact structure defined on $T^1M$: 
    \[
    S(T^1M)=E^{ss}\oplus E^s\oplus E^u\oplus E^{uu}.
    \]
\end{lemma}
\begin{proof}
Define the invariant subbundles 
\begin{flalign*}
&\begin{aligned}
\, \, \,\,\, \hspace{0.5cm} P_A^u(x,v)&=\{(w,w)\in S(x,v): w \in A(x,v)\},\\
\, \, \,\,\, P_B^u(x,v)&=\left\{\left(w,\frac{1}{2}w\right)\in S(x,v): w \in B(x,v)\right\},\\
\, \, \,\,\, P_A^s(x,v)&=\{(w,-w)\in S(x,v): w \in A(x,v)\},\\
\, \, \,\,\, P_B^s(x,v)&=\left\{\left(w,-\frac{1}{2}w\right)\in S(x,v): w \in B(x,v)\right\}.\\
\end{aligned}&&
\end{flalign*}

    Proceeding as in Lemma \ref{gamma is partially hyperbolic} one can conclude that
    \begin{align*}
        E^{uu}(x,v)=P_A^u(x,v), \,\,\, &E^{ss}(x,v)=P_A^s(x,v),\\
        E^u(x,v)=P_B^u(x,v),\,\,\, &E^s(x,v)=P_B^s(x,v).
    \end{align*}
\end{proof}
We are now going to define the family of cones that are invariant under the geodesic flow. In essence, the cones we are considering are cones around the strong stable and strong unstable spaces given by Lemma \ref{four invariant bundles}. In the next section, we show that the same family of cones is invariant under the geodesic flow of the deformed metric. Although our deformation does not make a big change on the spaces with smaller curvature, the computations to show the cone invariance are much more complicated. 
\begin{remark}
    Remember that if $E\subset TM$ is a subbundle, by denoting by $Pr_E$ the projection to the subspace $E$, we can define the quantity $\Uptheta_E(v)=\frac{g(P_E(v),P_E(v))}{g(v,v)}$ which is equal to the square of the cosine of the angle between the vector $v$ and the space $E$.
\end{remark}
The family of invariant cone fields we are going to consider is the following
\[
\Scale[0.92] {C(v,P_A^{u,s}(x,v),c)=\left\{ (\xi, \eta) \in S_{(x,v)}T^1M: \Uptheta_A^{u,s}(\xi, \eta)= \dfrac{\wh{g}(Pr_{P_A^{u,s}(x,v)}(\xi,\eta),Pr_{P_A^{u,s}(x,v)}(\xi,\eta))}{\wh{g}((\xi,\eta),(\xi,\eta))}\geq c\right\}},
\]
here $\wh{g}$ is the Sasaki metric. By the remark above, it is enough to see that this amount is increasing to guarantee the invariance. For the geodesic flow $g_t$, we want to compute
\[
\dfrac{d}{dt}\Uptheta_{A}^{u,s}((dg_t)_{(x,v)}(\xi,\eta)).
\]
We have that
\[
(dg_t)_{(x,v)}(\xi,\eta) = (J(t),J'(t)),
\]
where $J$ is the Jacobi field along the geodesic $\gamma:=\pi(g_t(x,v))$ with initial conditions $J(0)=\xi$ and $J'(0)=\eta$. For simplicity of notation, let us denote $J(t)=\xi$ and $J'(t)=\eta$, so they are related by the equations (see also (\ref{Jacobi}))
\[
\begin{cases}
    \xi'&=\eta\\
    \eta'&=-R(\xi,\gamma')\gamma'
\end{cases}.
\]
Hence, we can write
\[
\Uptheta_A^{u,s}(\xi, \eta)=\frac{g(\xi_A\pm\eta_A,\xi_A\pm\eta_A)}{g(\xi,\xi)+g(\eta,\eta)}.
\]
We calculate only $\dfrac{d}{dt}\Uptheta_A^u(\xi,\eta)$, because for $\Uptheta_A^s(\xi,\eta)$ it is analogous. The next Lemma is also proved in section 4.1.1 of \cite{CARNEIRO&PUJALS14}. We present the proof here because the computations will be used in the next sections.

\begin{lemma}
    With the notation above,
    \[
    \dfrac{d}{dt}\Uptheta_A^u(\xi,\eta)>0.
    \]
\end{lemma}
\begin{proof}
The straightforward computation of the derivative gives us the following equality
\begin{align*}
    \dfrac{d}{dt}\Uptheta_A^u(\xi,\eta)&=\frac{2g((\xi_A)'+(\eta_A)',\xi_A+\eta_A)(g(\xi,\xi)+g(\eta,\eta))}{(g(\xi,\xi)+g(\eta,\eta))^2}\\
    &-2\frac{(g(\xi',\xi)+g(\eta',\eta))g(\xi_A+\eta_A,\xi_A+\eta_A)}{(g(\xi,\xi)+g(\eta,\eta))^2}
\end{align*}
A priori $(\xi_A)'=\xi_{A'}+\xi_A'$, but $\xi_{A'}=0$ since $A$ is parallel. Then we write
\begin{align*}
    \dfrac{d}{dt}\Uptheta_A^u(\xi,\eta)&=2\frac{g(\xi_A'+\eta_A',\xi_A+\eta_A)}{g(\xi,\xi)+g(\eta,\eta)}\\
                                    &-2\frac{(g(\xi',\xi)+g(\eta',\eta))g(\xi_A+\eta_A,\xi_A+\eta_A)}{(g(\xi,\xi)+g(\eta,\eta))^2}.
\end{align*}
Because $\xi$ and $\eta$ satisfy the following equations (remember that $\xi$ is a Jacobi field)

\[\begin{cases}
    \xi'&=\eta, \\
    \xi_A'&=Pr_A(\xi')=Pr_A(\eta)=\eta_A,\\
    \eta'&=-R(\xi,\gamma')\gamma',\\
    \eta_A'&=Pr_A(\eta')=Pr_A(-R(\xi,\gamma')\gamma')=-(R(\xi,\gamma')\gamma')_A,
\end{cases}
\]
then
\begin{align}
\label{angle variation initial metric}
    \dfrac{d}{dt}\Uptheta_A^u(\xi,\eta)&=2\frac{g(\eta_A-(R(\xi,\gamma')\gamma')_A,\xi_A+\eta_A)}{g(\xi,\xi)+g(\eta,\eta)}\nonumber\\
                                    &-2\frac{g(\xi_A+\eta_A,\xi_A+\eta_A)}{(g(\xi,\xi)+g(\eta,\eta))^2}(g(\eta,\xi)+g(-R(\xi,\gamma')\gamma',\eta)).
\end{align}
Because of the properties $R(X,v)v=-\frac{1}{4}X_B-X_A$ and $A$ is orthogonal to $B$, we have that
\[
(R(\xi,\gamma')\gamma')_A=-\xi_A
\]
and
\[
g(-R(\xi,\gamma')\gamma',\eta)=\frac{1}{4}g(\eta_B,\xi_B)+g(\eta_A,\xi_A).
\]
So we get
\begin{align}
\label{final angle variation for initial metric}
    \dfrac{d}{dt}\Uptheta_A^u(\xi,\eta)&=2\frac{g(\eta_A+\xi_A,\xi_A+\eta_A)}{g(\xi,\xi)+g(\eta,\eta)}\nonumber\\
                                    &-2\frac{g(\xi_A+\eta_A,\xi_A+\eta_A)}{(g(\xi,\xi)+g(\eta,\eta))^2}\left(g(\eta,\xi)+\frac{1}{4}g(\eta_B,\xi_B)+g(\eta_A,\xi_A)\right)\nonumber\\
                                    &=2\frac{g(\eta_A+\xi_A,\xi_A+\eta_A)}{(g(\xi,\xi)+g(\eta,\eta))^2}\left(g(\xi,\xi)+g(\eta,\eta)-g(\eta,\xi)-\frac{1}{4}g(\eta_B,\xi_B)-g(\eta_A,\xi_A)\right) .                     
\end{align}

Write $M:=2\frac{g(\eta_A+\xi_A,\xi_A+\eta_A)}{(g(\xi,\xi)+g(\eta,\eta))^2}>0$. Then,
\begin{align*}
    \dfrac{d}{dt}\Uptheta_A^u(\xi,\eta)&=M\left(g(\xi,\xi)+g(\eta,\eta)-g(\eta,\xi)-\frac{1}{4}g(\eta_B,\xi_B)-g(\eta_A,\xi_A)\right)\\
                                &=M\Bigg(g(\xi_A,\xi_A)+g(\xi_B,\xi_B)+g(\eta_A,\eta_A)+g(\eta_B,\eta_B)\Bigg.\\
                                &\Bigg.-g(\xi_A,\eta_A)-g(\xi_B,\eta_B)-\frac{1}{4}g(\xi_B,\eta_B)-g(\xi_A,\eta_A)\Bigg)\\
                                &=M\Bigg(g(\xi_A,\xi_A)-2g(\xi_A,\eta_A)+g(\eta_A,\eta_A)+g(\xi_B,\xi_B)-\frac{5}{4}g(\xi_B,\eta_B)+g(\eta_B,\eta_B)\Bigg)\\
                                &=M\left(g(\xi_A-\eta_A,\xi_A-\eta_A)+g\left(\xi_B-\frac{5}{8}\eta_B,\xi_B-\frac{5}{8}\eta_B\right)+\frac{39}{64}g(\eta_B,\eta_B)\right).
\end{align*}
The last equality is trivially positive. 
\end{proof}


\subsection{Cone invariance for the deformed metric}
To prove that the geodesic flow $\til{g}_t$ is partially hyperbolic, we follow the same idea as in \cite{CARNEIRO&PUJALS14}. We prove that there is a family of invariant cones around the candidates to strong stable and unstable bundles. The strategy is to prove the invariance of the cones for geodesics that are parallel to the central geodesic and geodesics that are “$s$-almost parallel” by computing the angle variation. For transversal geodesics, we prove that by shrinking the deformed region we can guarantee that these geodesics spend just a small time in the deformed region in comparison to the time they spend outside, thus we can conclude the cone invariance for those geodesics too. Here we just need to do the computations for parallel and almost parallel geodesics with controlled components in the $s$-direction, once their arguments for the transversal geodesics do not depend on the type of deformation made, so it automatically works in our case.

As before, we are interested in the angle variation between the orbit of the derivative of the geodesic flow and a subbundle of the contact structure for the new metric. Again, it is given by the angle of the corresponding Jacobi field and the subbundle. For this subsection, we will use the same notation for the Jacobi fields of the new metric, but we need to keep in mind that it is not a Jacobi field for the initial metric.

Let us again fix some notation. Remember that the angle we are interested in is given by $\tilde{\Uptheta}_{A^{u,s}}((d\tilde{g}_t)_{(x,v)}(\xi,\eta))$. Write $\alpha(t)=\pi\circ \tilde{g}_t(x,v)$ and $(d\tilde{g}_t)_{(x,v)}(\xi,\eta)=(\xi(t),\eta(t))$, where $\xi(t)$ and $\eta(t)$ satisfy
\[
\begin{cases}
    \xi'(t)=\eta(t),\\
    \eta'(t)=-\tilde{R}(\xi(t),\alpha'(t))\alpha'(t).
\end{cases}
\]
Once again we write
\[
\tilde{\Uptheta}_A^{u,s}(\xi, \eta)=\frac{\tilde{g}(\xi_A\pm\eta_A,\xi_A\pm\eta_A)}{\tilde{g}(\xi,\xi)+\tilde{g}(\eta,\eta)}.
\]
\begin{remark}
    We need to be careful with the differences in the metric phenomena. For instance, we will deal with a geodesic $\alpha$ for the deformed metric. Then, from equation \eqref{conformal connection} we get
    \[
    0=\dfrac{\tilde{D}}{dt}\alpha'=\dfrac{D}{dt}\alpha'+\alpha'(h)\alpha'-\frac{1}{2}g(\alpha',\alpha')\nabla h.
    \]
    So, if the geodesic $\alpha$ lies outside the deformed region, then it is a geodesic for the initial metric, but inside this region, it may not be a geodesic for the initial metric. The same remark can be made about Jacobi fields. For simplicity and to differentiate the operators $\dfrac{\tilde{D}}{dt}$ and $\dfrac{D}{dt}$, we will use the notation $\tilde{D}_t=\dfrac{\tilde{D}}{dt}$ and $D_t=\dfrac{D}{dt}$. With this notation
    \[
    \begin{cases}
        \tilde{D}_t\xi&=\eta,\\
        \tilde{D}_t\eta&=-\tilde{R}(\xi,\alpha')\alpha'.
    \end{cases}
    \]
\end{remark}
 Of course, if the geodesic $\alpha$ does not cross the deformed region, then the angle variation is the same as before.

\subsubsection{Cone invariance for parallel geodesics}
First, we will evaluate the variation of the angle for geodesics that point in the same direction as the central geodesic, i.e. $\alpha'=(\alpha^0,0,\dots,0)$. In general, we can write
\begin{align*}
    \dfrac{d}{dt}\til\Uptheta_A^u(\xi,\eta)&=2\frac{\til g((\xi_A)'+(\eta_A)',\xi_A+\eta_A)}{\til g(\xi,\xi)+\til g(\eta,\eta)}\\
                                    &-2\frac{(\til g(\xi',\xi)+\til g(\eta',\eta))\til g(\xi_A+\eta_A,\xi_A+\eta_A)}{(\til g(\xi,\xi)+\til g(\eta,\eta))^2}\\
                                &=2\frac{\til g(\xi_{A'}+\xi_A'+\eta_{A'}+\eta_A',\xi_A+\eta_A)}{\til g(\xi,\xi)+\til g(\eta,\eta)}\\
                                    &-2\frac{(\til g(\xi',\xi)+\til g(\eta',\eta))\til g(\xi_A+\eta_A,\xi_A+\eta_A)}{(\til g(\xi,\xi)+\til g(\eta,\eta))^2}\\
                                &=2\frac{\til g(\xi_A'+\eta_A',\xi_A+\eta_A)}{\til g(\xi,\xi)+\til g(\eta,\eta)}+2\frac{\til g(\xi_{A'}+\eta_{A'},\xi_A+\eta_A)}{\til g(\xi,\xi)+\til g(\eta,\eta)}\\
                                    &-2\frac{(\til g(\xi',\xi)+\til g(\eta',\eta))\til g(\xi_A+\eta_A,\xi_A+\eta_A)}{(\til g(\xi,\xi)+\til g(\eta,\eta))^2}.
\end{align*}
Write $N:=\til g(\xi,\xi)+\til g(\eta,\eta)$ to make our notation simpler.
\begin{align*}
    \dfrac{d}{dt}\til\Uptheta_A^u(\xi,\eta)&=2\frac{\til g(\xi_A'+\eta_A',\xi_A+\eta_A)}{N}+2\frac{\til g(\xi_{A'}+\eta_{A'},\xi_A+\eta_A)}{N}\\
                                    &-2\frac{(\til g(\xi',\xi)+\til g(\eta',\eta))\til g(\xi_A+\eta_A,\xi_A+\eta_A)}{N^2}\\
                                    &=2\frac{\til g(\eta_A- (\til R(\xi,\alpha')\alpha')_A,\xi_A+\eta_A)}{N}+2\frac{\til g(\xi_{A'}+\eta_{A'},\xi_A+\eta_A)}{N}\\
                                    &-2\frac{(\til g(\eta,\xi)+\til g(-\til R(\xi,\alpha')\alpha',\eta))\til g(\xi_A+\eta_A,\xi_A+\eta_A)}{N^2}.
\end{align*}
The strategy to show that this quantity is positive is to approximate it by \eqref{angle variation initial metric} plus terms that can be made small or do not change its sign. Since $h$ does not depend on $x_0$, we have that $\alpha'(h)=0$. We first deal with the term $\til{g}(\eta_A-(\til{R}(\xi,\alpha')\alpha')_A,\xi_A+\eta_A)$: from the equation \eqref{conformal curvature} and by the definition of the space $A$ we get  the following:

\begin{align*}
-(\til{R}(\xi,\alpha')\alpha')_A&=-(R(\xi,\alpha')\alpha')_A+\frac{1}{2}Hess(h)(\alpha',\alpha')\xi_A\\
&+\frac{1}{2}\norm{\alpha'}^2(\nabla_{\xi}\nabla h)_A+\frac{1}{4}\norm{\alpha'}^2\norm{\nabla h}^2\xi_A-(\xi h)\norm{\alpha'}^2(\nabla h)_A.
\end{align*}
Notice that by the estimates in Lemma \ref{estimates on h}
\begin{itemize}
    \item There exists $C_1>0$ such that $|Hess(h)(\alpha',\alpha')|\leq C_1\ve^2$.
    \item There exists $C_2>0$ such that $\frac{1}{4}\norm{\alpha'}\norm{\nabla h}^2\leq C_2\ve^2$.
    \item There exists $C_3>0$ such that $|(\xi h)\norm{\alpha'}^2\norm{(\nabla h)_A}\leq C_3\norm{\xi}\ve^2$.
\end{itemize}
It implies that there exists $C>0$ such that we can write
\begin{equation}
\label{aproximation curvature projected}
-(\til{R}(\xi,\alpha')\alpha')_A=-(R(\xi,\alpha')\alpha')_A+\frac{1}{2}\norm{\alpha'}^2(\nabla_{\xi}\nabla h)_A+v_A+w_A,
\end{equation}
with $v_A,w_A\in A$ such that $\norm{v_A}\leq C\ve^2$ and $\norm{w_A}\leq C\norm{\xi}\ve^2$. We claim that we can also control the norm of $(\nabla_\xi\nabla h)_A$. First, remember that since $A$ is parallel for the initial metric, we can write

\[
(\nabla_\xi\nabla h)_A=\nabla_\xi(\nabla h)_A.
\]
Now,
\[
\nabla_\xi(\nabla h)_A=\nabla_{\xi}\left(\partial_ih\left(\frac{\partial}{\partial x^i}\right)_A\right)=\xi(\partial_ih)\left(\frac{\partial}{\partial x^i}\right)_A+\partial_ih\nabla_\xi\left(\frac{\partial}{\partial x^i}\right)_A.
\]
Because the Christoffel symbols vanish at $x=0$, there exists a constant $D_1>0$ such that the second term has a norm less than or equal to $D_1\norm{\xi}\ve^2$. For the first term also by Lemma \ref{estimates on h} we get that there exists $D_2>0$ such that for all $i\neq s$
\[
\norm{\xi(\partial_ih)\left(\frac{\partial}{\partial x^i}\right)_A}\leq D_2\norm{\xi}\ve.
\]
For $i=s$
\begin{equation}
\label{norm covariant derivative projected}
\norm{\xi(\partial_sh)\left(\frac{\partial}{\partial x^s}\right)_A}=\norm{\sum_{j }\xi_j\partial_{js}^2h\left(\frac{\partial}{\partial x^s}\right)_A}\leq \norm{\sum_{j\neq s }\xi_j\partial_{js}^2h\left(\frac{\partial}{\partial x^s}\right)_A}+\norm{\xi_s\partial_{ss}^2h\left(\frac{\partial}{\partial x^s}\right)_A}.
\end{equation}
Then, there exists $D_3>0$ such that
\[
\norm{\sum_{j\neq s }\xi_j\partial_{js}^2h\left(\frac{\partial}{\partial x^s}\right)_A}\leq D_3\norm{\xi}\ve.
\]
For the last term of \eqref{norm covariant derivative projected} notice that $\left(\frac{\partial}{\partial x^s}\right)_{A(\gamma,\gamma')}=0$, then, since $\alpha$ is parallel, by continuity of the subbundle $A$ there exists $D_4>0$ such that
\[
\norm{\xi_s\partial_{ss}^2h\left(\frac{\partial}{\partial x^s}\right)_A}\leq D_4|\xi_s|\ve.
\]
We conclude that there exists $D>0$ such that

\[
\norm{(\nabla_\xi\nabla h)_A}\leq D\norm{\xi}\ve.
\]
Back to equality \eqref{aproximation curvature projected} there exists a constant $\til{C}>0$ such that
\begin{equation}
\label{curvature projected plus small vectors}
    -(\til{R}(\xi,\alpha')\alpha')_A=-(R(\xi,\alpha')\alpha')_A+v_A+w_A,
\end{equation}
with $v_A,w_A\in A$ such that $\norm{v_A}\leq \til{C}\ve^2$ and $\norm{w_A}\leq \til{C}\norm{\xi}\ve$. Finally, using $(R(\xi,\alpha')\alpha')_A=-\xi_A$ we can write
\[
\til{g}(\eta_A-(\til{R}(\xi,\alpha')\alpha')_A,\xi_A+\eta_A) = \til{g}(\xi_A+\eta_A,\xi_A+\eta_A)+\til{g}(v_A,\xi_A+\eta_A)+\til{g}(w_A,\xi_A+\eta_A).
\]
We have shown that $2\frac{\til g(\eta_A- (\til R(\xi,\alpha')\alpha')_A,\xi_A+\eta_A)}{N}$ can be made arbitrarily close to $2\frac{\til g(\xi_A+\eta_A,\xi_A+\eta_A)}{N}$, which is similar to the first term that appears in \eqref{final angle variation for initial metric}. We now deal with $\til g(-\til R(\xi,\alpha')\alpha',\eta)$: from the relation \eqref{conformal curvature} and the fact that $\eta$ is orthogonal to $\alpha'$ we get

\begin{align}
\label{approximation curvature NOT projected}
    \til g(-\til R(\xi,\alpha')\alpha',\eta) &=\til g(-R(\xi,\alpha')\alpha',\eta)+\frac{1}{2}Hess(h)(\alpha',\alpha')\til{g}(\xi,\eta)\nonumber \\
    &+\frac{1}{2}\norm{\alpha'}^2Hess(h)(\xi,\eta)+\frac{1}{4}\norm{\alpha'}^2\norm{\nabla h}^2\til g(\xi,\eta)+\frac{1}{4}(\xi h)\norm{\alpha'}^2(\eta h).
\end{align}
Proceeding similarly to before and applying the estimates from Lemma \ref{estimates on h}, we have the following:
\begin{itemize}
    \item There exists $C_1>0$ such that $\frac{1}{2}|Hess(h)(\alpha',\alpha')|\leq C_1\ve^2$.
    \item There exists $C_2>0$ such that $\frac{1}{4}\norm{\alpha'}\norm{\nabla h}^2 \leq C_2\ve^2$.
    \item There exists $C_3>0$ such that $\frac{1}{4}\norm{\alpha'}^2|(\xi h)(\eta h)|\leq C_3 \norm{\xi}\norm{\eta}\ve^2$. 
\end{itemize}
We also can have some control on $Hess(h)(\xi,\eta)$: there exists $C_4$ such that
\[
|Hess(h)(\xi,\eta)-\xi_s\eta_s\partial_{ss}^2h|\leq \sum_{(i,j)\neq(s,s)}|\xi^i\eta^j||\partial_{ij}^2h-\Gamma_{ij}^k\partial_kh|\leq C_4\norm{\xi}\norm{\eta}\ve.
\]
Putting everything together, we guarantee that there exists $C>0$ such that we can rewrite equation \eqref{approximation curvature NOT projected} as
\begin{equation}
\label{curvature plus small numbers}
\til g(-\til R(\xi,\alpha')\alpha',\eta) = \til g(-R(\xi,\alpha')\alpha',\eta) -\frac{1}{2}\norm{\alpha'}^2\xi_s\eta_s\partial_{ss}^2h+K_1+K_2,
\end{equation}
with $|K_1|\leq C\ve^2$ and $|K_2|\leq C\norm{\xi}\norm{\eta}\ve$. The last term we need to deal with is $\til g(\xi_{A'}+\eta_{A'},\xi_A+\eta_A)$, but record that from our notation $\xi_{A'}=(\til D_t Pr_A)\eta$ and $\til D_t$ is $\ve$-close to $D_t$ for which we had $D_tPr_A=0$ by the locally symmetric assumption. It implies that there exists $E>0$ such that $\til g(\xi_{A'}+\eta_{A'},\xi_A+\eta_A)\leq E\norm{\xi_A+\eta_A}\ve$. All the above estimates imply, by mimicking what we have done in equation \eqref{final angle variation for initial metric} and using \protect{\cite[Lemma~2.10]{CARNEIRO&PUJALS14}}, that there exist constants $K>0$ and $K'>0$ such that we can write the angle variation as

\begin{align}
\label{final angle variation for parallel geodesics}
    \dfrac{d}{dt}\til\Uptheta_A^u(\xi,\eta)=2\frac{\til g(\eta_A+\xi_A,\xi_A+\eta_A)}{(\til g(\xi,\xi)+\til g(\eta,\eta))^2}&\Big(\til g(\xi,\xi)+\til g(\eta,\eta)-\til g(\eta,\xi)-\frac{1}{4}\til g(\eta_B,\xi_B)-\til g(\eta_A,\xi_A)\nonumber \\
    &-\frac{1}{2}\norm{\alpha'}^2\xi_s\eta_s\partial_{ss}^2h\Big)+U_1+U_2+U_3+U_4,
\end{align}
with
\begin{itemize}
    \item $|U_1|\leq K\frac{\norm{\xi_A+\eta_A}}{N}\ve^2\leq K'\ve^2$.
    \item $|U_2|\leq K\frac{\norm{\xi_A+\eta_A}\norm{\xi}}{N}\ve\leq K'\ve$.
    \item $|U_3|\leq K\frac{\norm{\xi_A+\eta_A}^2}{N^2}\ve^2\leq K'\ve^2$.
    \item $|U_4|\leq K\frac{\norm{\xi_A+\eta_A}\norm{\xi}\norm{\eta}}{N^2}\ve\leq K'\ve$.
\end{itemize}
Thus, it is sufficient to prove that the quantity inside the parentheses is positive and then choose $\ve>0$ sufficiently small such that the expression \eqref{final angle variation for parallel geodesics} is positive. To see that, notice that since $\alpha'$ is an unit-speed geodesic for the metric $\til g$, then $\norm{\alpha'}^2=g(\alpha',\alpha')=\mrm{e}^{-h}$ does not exceed $\mrm{e}^{2\ve^4}$. It follows from the construction of the function $h$ that 
\[
-\frac{1}{2}\norm{\alpha'}^2\partial_{ss}^2h\leq \frac{\mrm{e}^{2\ve^4}}{4}\leq \frac{1}{2}.
\]
If $\xi_s\eta_s\geq0$, then by the same computations for the initial metric we can conclude that the expression \eqref{final angle variation for parallel geodesics} is positive. If $\xi_s\eta_s<0$ then we must have 
\[
\frac{1}{2}\xi_s\eta_s\leq -\frac{1}{2}\norm{\alpha'}^2\partial_{ss}^2h\xi_s\eta_s.
\]
So, it is sufficient to show that the following function is positive 
\[
f(x):=\til g(\xi,\xi)+\til g(\eta,\eta)-\til g(\eta,\xi)-\frac{1}{4}\til g(\eta_B,\xi_B)-\til g(\eta_A,\xi_A)\nonumber -\frac{1}{2}\xi_s\eta_s.
\]
Notice that $f$ does not depend on $\ve$, so we prove that $f(0)>0$, and by shrinking the deformed neighborhood if necessary we guarantee that $f(x)>0$.

\begin{align*}
    f(0)&= g(\xi,\xi)+ g(\eta,\eta)- g(\eta,\xi)-\frac{1}{4} g(\eta_B,\xi_B)- g(\eta_A,\xi_A)\nonumber -\frac{1}{2}\xi_s\eta_s\\
        &=\sum_{i=1}^{n-1}(\xi_i^2+\eta_i^2)-\sum_{i=1}^{n-1}\xi_i\eta_i-\frac{1}{4}\sum_{i=r+1}^{n-1}\xi_i\eta_i-\sum_{i=1}^r\xi_i\eta_i-\frac{1}{2}\xi_s\eta_s\\
        &=\sum_{i=1}^r(\xi_i^2-2\xi_i\eta_i+\eta_i^2)+\sum_{\substack{i=r+1\\i\neq s}}^{n-1}\left(\xi_i^2-\frac{5}{4}\xi_i\eta_i+\eta_i^2\right)+\left(\xi_s^2-\frac{7}{4}\xi_s\eta_s+\eta_s^2\right)\\
        &=\sum_{i=1}^r(\xi_i-\eta_i)^2+\sum_{\substack{i=r+1\\i\neq s}}^{n-1}\left(\xi_i-\frac{5}{8}\eta_i\right)^2+\frac{39}{64}\sum_{\substack{i=r+1\\i\neq s}}^{n-1}\eta_i^2+\left(\xi_s-\frac{7}{8}\eta_s\right)^2+\frac{15}{64}\eta_s^2\\
        &>0.
\end{align*}
We have proven the following proposition:
\begin{proposition}
\label{conclusion of invariance for parallel geodesics}
    There exists $\ve>0$ such that the family of cones $C(v,P_A^{u,s}(x,v),c)$ is invariant along parallel geodesics that cross $U(\ve)$. 
\end{proposition}
\subsubsection{Cone invariance for almost parallel geodesics}
Suppose that $\alpha'=\alpha^i\dfrac{\partial}{\partial x^i}$ is such that $|\alpha^s|\leq \theta$. We will prove that the cone family is invariant along such geodesics if $\theta$ is small enough. The proof follows the same lines as the proof for parallel geodesics, the difference here is that for some estimates we cannot assume that $\alpha'$ has no components in other directions, but it is possible to obtain essentially the same estimates for the first derivatives of $h$ in the $\alpha'$ direction, and $\theta$ small will control the hessian of $h$ applied to $\alpha'$. Indeed, for an almost parallel geodesic, the same analysis holds but for terms that are controlled by $\theta$: by equation \eqref{conformal curvature} we can see that
\begin{align*}
-(\til R(\xi,\alpha')\alpha')_A=-(R(\xi,\alpha')\alpha')_A+v_A+w_A-\frac{1}{4}(\alpha'(h))^2\xi_A+\frac{1}{4}(\xi h)(\alpha'(h))(\nabla h)_A\\
-\frac{1}{2}Hess(h)(\alpha',\alpha')(h)\xi_A,
\end{align*}
With $v_A$ and $w_A$ as before. However, there exist $C_1>0$ and $C_2>0$ such that
\[
\norm{-\frac{1}{4}(\alpha'(h))^2\xi_A-\frac{1}{4}(\xi h)(\alpha'(h))(\nabla h)_A}\leq C_1\norm{\xi}\ve^2
\]
and
\[
|Hess(h)(\alpha',\alpha')|\leq C_1\norm{\alpha'}^2\ve+C_2\theta^2.
\]
So, without losing generality, there exists a constant $C>0$ such that we still can write the relation \eqref{curvature projected plus small vectors} for almost parallel geodesics plus a term that is controlled by $\theta$, i.e.
\begin{equation}
\label{curvature projected plus small vectors controlled by theta}
    -(\til{R}(\xi,\alpha')\alpha')_A=-(R(\xi,\alpha')\alpha')_A+v_A+w_A+u_A,
\end{equation}
where $\norm{u_A}\leq C\norm{\xi}\theta^2$. Analogously, there exists a constant $D>0$ such that we can rewrite the equation \eqref{curvature plus small numbers} for almost parallel geodesics as follows:
\begin{align*}
\til g(-\til R(\xi,\alpha')\alpha',\eta) &= \til g(-R(\xi,\alpha')\alpha',\eta) -\frac{1}{2}\norm{\alpha'}^2\xi_s\eta_s\partial_{ss}^2h+K_1+K_2+K_3\\
&-\frac{1}{4}(\alpha'(h))^2\til{g}(\xi,\eta)+\frac{1}{4}(\xi h)(\alpha'(h))(\eta h).
\end{align*}
Again, the extra terms satisfy the same properties as $K_1$ and $K_2$ and $|K_3|\leq D\norm{\xi}\norm{\eta}\theta^2$. Then we can write, without losing generality, the equation \eqref{final angle variation for parallel geodesics} for almost parallel geodesics plus a term controlled by $\theta^2$. Since we already proved that this expression is positive, the next Proposition follows:

\begin{proposition}
\label{conclusion cone invariance for s controlled geodesics}
    There exist numbers $\ve>0$ and $\theta>0$ such that the family of cones $C(v,P_A^{u,s}(x,v),c)$ is invariant along any geodesic $\alpha$ with $|\alpha^s|<\theta$ that crosses $U(\ve)$.
\end{proposition}
Notice that the above proposition works for any $\ve^*\leq\ve$ and $\theta^*\leq\theta$.
\subsubsection{Cone invariance for $s$-transversal geodesics}
Fix $\ve_1>0$ and $\theta_1>0$ given by Proposition \ref{conclusion of invariance for parallel geodesics} and Proposition \ref{conclusion cone invariance for s controlled geodesics}. Since $\Gamma_{ij}^k(t,0)=0$, then by relation \eqref{conformal christoffel symbols} we get that $\til{\Gamma}_{ij}^k(t,0)=0$ and also by the same relation we can find a constant $C>0$ such that $|\til{\Gamma}_{ij}^k(t,x)|\leq C\ve_1$. Then, denoting by $\til{G}$ the geodesic vector field for the metric $\til g$ we can find a constant $D>0$ such that
\[
\norm{\til G(v)-(v_1,...,v_n,0,...,0)}\leq D\ve_1.
\]
So, any geodesic $\alpha'(0) = (v_0,...,v_{n-1})$ in $U(\ve_1)$ is $\ve_1$-close to the curve $\beta(s)=(sv_0,...,sv_{n-1})$. If $|v_s|>\theta_1$, then for $s\geq\frac{2\ve_1}{\theta_1}$ the curve $\beta(s)$ escapes the neighborhood $U(2\ve_1)$, so the geodesic $\alpha$ escapes the neighborhood $U(\ve_1)$. Since $\theta_1$ is fixed, let $\ve_2<\ve_1$ such that we have $\frac{2\ve_2}{\theta_1}<\ve_1$. From here, the same arguments as in \cite{CARNEIRO&PUJALS14} Section 5.3.6. it implies
\begin{proposition}
\label{cone invariance for the deformed geodesic flow}
    There exists $\ve>0$ and $b>0$ such that if $h$ is supported in $U(\ve)$, the family of cones $C(v,P_A^{u,s}(x,v),b)$ is invariant under the action of $d\til g_t$.
\end{proposition}

\begin{corollary}
\label{the geodesic flow has dominated splitting}
    The geodesic flow $\til g_t: T^1M \rarrow T^1M$ admits a dominated splitting of the form $ST^1M=E^{uu}\oplus E^c\oplus E^{ss}$.
\end{corollary}

\subsubsection{Proof of Theorem \ref{there exist conformal partially hyperbolic}}
The proof of Theorem \ref{there exist conformal partially hyperbolic} goes as follows: We have proved in Corollary \ref{the geodesic flow has dominated splitting} that the geodesic flow $d\tilde{g}_t: ST^1M \rightarrow ST^1M$ is a symplectic flow with dominated splitting of the form  $ST^1M=E^s\oplus E^c\oplus E^u$. Now Lemma \ref{dominated splitting in symplectic} implies that the vectors in $E^{u}$ and $E^{s}$ are uniformly expanding and contracting, respectively, thus $\til g_t$ is partially hyperbolic by the cone criterion.

We conclude this section with two remarks:

\begin{remark}
    We do not know a priori if the new metric has conjugate points. The best we can expect to answer in this direction is the following: since $-\frac{1}{2}\partial_{ss}^2h\leq \frac{1}{4}$, $K(X,Y)\leq -\frac{1}{4}$, Lemma \ref{estimates on h} and the estimates on the Christoffel symbols for the initial metric, it follows that $\tilde K(X,Y)\leq M\ve$. So, the resultant sectional curvatures could be positive, but they are still controlled by $\varepsilon$ and can be made as small as we want. If we have uniform control over the amount of time a geodesic can spend in a region of possible positive curvature, then we can conclude the nonexistence of conjugate points by the arguments in \cite{GULLIVER1975}.
\end{remark}

\begin{remark}
    For our construction, we fixed the $s$-direction for which we perform the largest curvature deformation. However, under small adaptations to the arguments, it is possible to produce partially hyperbolic examples by deforming the curvature in several directions. To do this, consider some subcollection of indices $\mathcal{I}\subset \{r+1,...,n-1\}$ and let $h$ be given by
    \[
    h(x)=\sum_{s\in\mathcal{I}} h^s(x),
    \]
    where $h^s$ is the function constructed in Section \ref{Breaking the Anosov Property} for the $s$-direction. However, by doing so we certainly create regions of positive curvature. More precisely, in the Kähler setting we have that  $A(x,v)=\text{span}\{Jv\}$ and for some $s,l\in \mathcal{I}$ we may obtain 
    \[
    \til K(e_s(t),e_l(t))=K(e_s(t),e_l(t))-\frac{1}{2}\partial_{ss}h^s(0)-\frac{1}{2}\partial_{ll}h^l(0)=-\frac{1}{4}+\frac{1}{4}+\frac{1}{4}=\frac{1}{2}.
    \]
\end{remark}

\section{Proof of Theorem \ref{ergodic partially hyperbolic geodesic flow}}
\label{Proof of Theorem B}
The proof of Theorem \ref{ergodic partially hyperbolic geodesic flow} is similar to Theorem \ref{there exist conformal partially hyperbolic}. We work in the same context and with the same tubular neighborhood of a closed geodesic. The difference lies in the different deformations we make to have better control of the resulting sectional curvatures. Here we present only the estimates on the curvature, since the proof of partial hyperbolicity is the same as presented in the previous section, as the arguments do not depend on the deformation itself, but on the curvature tensor estimates, and in this case the same estimates hold. The deformation goes as follows: let $g_{ij}(t,x)$ and $g^{ij}(t,x)$ be the components of the initial locally symmetric metric and its inverse in the Fermi coordinates defined in the previous section. Considering $h$ as the same function as before, the new metric $\til{g}$ is given by
$$\til{g}_{00}(t,x)=e^hg_{00}(t,x)$$
$$\til{g}_{ij}(t,x)=g_{ij}(t,x), \,\, (i,j)\neq (0,0).$$
The new deformation is almost a conformal deformation of the initial metric, however, we only multiply (by a conformal factor) the components of the metric in the geodesic direction. By doing so, we prevent the curvature from becoming positive, as is shown above. Notice that in this case we have $\tilde{g}^{ij}=g^{ij}$ for $(i,j)\neq (0,0)$ and $\tilde{g}^{00}=e^{-h}g_{00}$.

\begin{remark}
\label{remark derivative g00}
  We have the following identities for the derivatives of the new metric:
\begin{enumerate}
    \item $\partial_i(\til{g}_{00})=\partial_ih\til{g}_{00}+e^h\partial_ig_{00}$.
    \item $\partial^2_{ij}(\til{g}_{00})=\partial_{ij}^2h\til{g}_{00}+\partial_ih\partial_jh\til{g}_{00}+e^h\partial_ih\partial_jg_{00}+e^h\partial_jh\partial_ig_{00}+e^h\partial_{ij}^2g_{00}$.
\end{enumerate}
\end{remark}
 We will first evaluate the Christoffel symbols of the new metric. We will split it into cases
\begin{enumerate}
    \item $i,j\neq 0$
    \begin{enumerate}
        \item $k=0$
        \begin{align*}
            \til{\Gamma_{ij}^0}&=\frac{\til{g}^{l0}}{2}(\partial_i\til{g}_{jl}+\partial_j\til{g}_{il}-\partial_l\til{g}_{ij})\\
            &=\sum_{l\neq0}\frac{\til{g}^{l0}}{2}(\partial_i\til{g}_{jl}+\partial_j\til{g}_{il}-\partial_l\til{g}_{ij})+\frac{\til{g}^{00}}{2}(\partial_i\til{g}_{j0}+\partial_j\til{g}_{i0}-\partial_0\til{g}_{ij})\\
            &=\sum_{l\neq0}\frac{{g}^{l0}}{2}(\partial_i g_{jl}+\partial_j g_{il}-\partial_l g_{ij})+\frac{e^{-h}g^{00}}{2}(\partial_i g_{j0}+\partial_j g_{i0}-\partial_0 g_{ij}).
        \end{align*}
        \item $k\neq 0$
        \[
        \til{\Gamma_{ij}^k}=\Gamma_{ij}^k.
        \]
    \end{enumerate}
    \item $i=0$, $j\neq 0$
    \begin{enumerate}
        \item $k=0$
        \[
        \til{\Gamma}_{0j}^0=\sum_{l\neq 0}\frac{g^{0l}}{2}(\partial_0g_{jl}+\partial_jg_{0l}-\partial_lg_{0j})+e^{-h}\frac{g^{00}}{2}(\partial_0g_{j0}+e^h\partial_jg_{00}-\partial_0g_{0j})+\frac{g^{00}}{2}g_{00}\partial_jh.
        \]
        \item $k\neq 0$
        \[
        \til{\Gamma}_{0j}^k=\sum_{l\neq 0}\frac{g^{kl}}{2}(\partial_0g_{jl}+\partial_jg_{0l}-\partial_lg_{0j})+\frac{g^{k0}}{2}(\partial_0g_{j0}+e^h\partial_jg_{00}-\partial_0g_{0j})+\frac{g^{k0}}{2}g_{00}e^h\partial_jh.
        \]
    \end{enumerate}
    \item $i=j=0$
    \begin{enumerate}
        \item $k=0$
        \[
        \til{\Gamma}_{00}^0=\sum_{l\neq 0}\left(\frac{g^{0l}}{2}(\partial_0g_{0l}+\partial_0g_{0l}-e^h\partial_lg_{00})-e^h\frac{g^{0l}}{2}g_{00}\partial_lh\right)+\frac{g^{00}}{2}\partial_0g_{00}.
        \]
        \item $k\neq 0$
        \[
        \til{\Gamma}_{00}^k=\sum_{l\neq 0}\left(\frac{g^{kl}}{2}(\partial_0g_{0l}+\partial_0g_{0l}-e^h\partial_lg_{00})-e^h\frac{g^{kl}}{2}g_{00}\partial_lh\right)+\frac{g^{k0}}{2}\partial_0g_{00}.
        \]
    \end{enumerate}
\end{enumerate}

The properties of $h$ and the computations above imply that $\til{\Gamma}_{ij}^k(t,0)=\Gamma_{ij}^k(t,0)=0$ and $\gamma$ is still a geodesic for the metric $\til g$. Furthermore, we also have by the properties of $h$, that the contribution of the new Christoffel symbols to the curvature $\til{R}_{s00s}$ that appears in the equation \eqref{NEW curvature in coordinates} below is close to the contribution of the initial Christoffel symbols and with the same sign inside the deformed region. To see this, consider, for instance, the functions $\til f_{nm}(x)=\til\Gamma_{0s}^m\til\Gamma_{0s}^n$ and $f_{nm}(x)=\Gamma_{0s}^m\Gamma_{0s}^n$. A straightforward computation shows that at $x=0$ we get that $\til f_{nm}$ and $f_{nm}$ are equal up to the second derivative. Indeed, we have $\til{\Gamma}_{0s}^m\til{\Gamma}_{0s}^n\leq\Gamma_{0s}^m\Gamma_{0s}^n$ by computing higher-order derivatives of the same functions for some indices $(m,n)$ and shrinking the deformed region if necessary. By doing the same kind of comparison with the terms $\til\Gamma_{00}^m\til\Gamma_{ss}^n$ we get that $-\til\Gamma_{ss}^m\til\Gamma_{00}^n\leq-\Gamma_{ss}^m\Gamma_{00}^n$.
\subsection{Curvature over the central geodesic}
Remember that from Lemma \ref{curvature in coordinates} we can write
\begin{equation}
\label{NEW curvature in coordinates}
\til{R}_{ijkr}=\til{R}_{ijk}^l\til{g}_{lr}=\frac{1}{2}(\partial_{jr}^2\til{g}_{ik}+\partial_{ik}^2\til{g}_{jr}-\partial_{ir}^2\til{g}_{jk}-\partial_{jk}^2\til{g}_{ir})+\til{g}_{mn}(\til{\Gamma}_{ki}^m\til{\Gamma}_{jr}^n-\til{\Gamma}_{ri}^m\til{\Gamma}_{jk}^n).
\end{equation}
In particular for $j=k$ and $i=r$ we get
\begin{equation*}
\label{sectional curvature}
    \til{R}_{ikki}=\frac{1}{2}(\partial_{ki}^2\til{g}_{ik}+\partial_{ik}^2\til{g}_{ki}-\partial_{ii}^2\til{g}_{kk}-\partial_{kk}^2\til{g}_{ii})+\til{g}_{mn}(\til{\Gamma}_{ki}^m\til{\Gamma}_{ki}^n-\til{\Gamma}_{ii}^m\til{\Gamma}_{kk}^n).
\end{equation*}
For $k=0$ at $x=0$, we get by Remark \ref{remark derivative g00}
\begin{align*}
    \til{R}_{i00i}&=\frac{1}{2}(\partial_{0i}^2\til{g}_{i0}+\partial_{i0}^2\til{g}_{0i}-\partial_{ii}^2\til{g}_{00}-\partial_{00}^2\til{g}_{ii})\\
    &=\frac{1}{2}(\partial_{0i}^2g_{i0}+\partial_{i0}^2g_{0i}-e^h\partial_{ii}^2g_{00}-\partial_{00}^2g_{ii})\\
    &-\frac{1}{2}\partial_{ii}^2he^hg_{00}-\frac{1}{2}(\partial_ih)^2e^hg_{00}-e^h\partial_ih\partial_ig_{00}\\
    &=R_{i00i}-\frac{1}{2}\partial^2_{ii}h.
\end{align*}
Then, for $i=s$ we get $\til{R}_{s00s}(t,0)=0$ and for $i\neq s$ $\til{R}_{i00i}(t,0)=R_{i00i}(t,0)<0$.
For $Y=Y^iX_i$ a vector field orthonormal to $X_s$ at $x=0$ such that $Y\neq X_0$ we get 
 \begin{align*}
     \til{R}(X_s,Y,Y,X_s)=Y^iY^j\til{R}_{sijs}&=\frac{1}{2}\sum_{(i,j)\neq(0,0)}Y^iY^j(\partial_{is}^2g_{sj}+\partial_{sj}^2g_{is}-\partial_{ss}^2g_{ij}-\partial_{ij}^2g_{ss})\\
     &+\frac{1}{2}Y_0^2(\partial_{0s}^2g_{s0}+\partial_{s0}^2g_{0s}-\partial_{ss}^2\til{g}_{00}-\partial_{00}^2g_{ss})\\
     &=R(X_s,Y,Y,X_s)-\frac{1}{2}\partial_{ss}^2hY_0^2\\
     &\leq-\frac{1}{4}+\frac{1}{4}Y_0^2\\
     &<0   .
 \end{align*}
 The inequality above shows that every sectional curvature is negative over the central geodesic, except for the sectional curvature of the plane $\Pi=\text{span}\{\gamma',e_s\}$, which is zero.
 \subsection{Curvature outside the central geodesic}
 Notice that outside the central geodesic, we can write
 \[
 \til{R}_{s00s}=R_{s00s}^*-\frac{1}{2}\partial_{ss}^2he^hg_{00}-\frac{1}{2}(\partial_sh)^2e^hg_{00}-e^h\partial_sh\partial_sg_{00},
 \]
 where
 \[
 R_{s00s}^*=\frac{1}{2}(\partial_{0s}^2g_{s0}+\partial_{s0}^2g_{0s}-e^h\partial_{ss}^2g_{00}-\partial_{00}^2g_{ss})+\til{g}_{mn}(\til{\Gamma}_{0s}^m\til{\Gamma}_{0s}^n-\til{\Gamma}_{ss}^m\til{\Gamma}_{00}^n),
 \]
 which is similar to the expression for $R_{s00s}$ but for the terms $-e^h\partial_{ss}^2g_{00}$ and the small perturbation of the Christoffel symbols, which does not change the sign of the contribution of this part as mentioned before. In Fermi coordinates we have that $g_{s0}\neq0$ and $g_{ss}\neq 1$ for $x\neq 0$, we can write
 \[
 \til{K}(X_s,X_0)=\til{K}\left(\frac{X_s}{g_{ss}},Y\right),
 \]
 where $Y$ is the orthonormal (to $X_s$) vector field we obtain by applying Gram–Schmidt process for the metric $\til{g}$:
 \[
 Y=\left(\frac{g_{ss}}{e^hg_{00}g_{ss}-g_{s0}^2}\right)X_0-\left(\frac{g_{s0}}{e^hg_{00}g_{ss}-g_{s0}^2}\right)X_s=\til{\alpha}X_0+\til{\beta}X_s.
 \]
 Notice that since $Y$ is a normal vector, we get $\til{\alpha}< 1$ (for $x\neq 0)$. Besides that, since $h\leq 0$ an easy computation shows that $\til{\alpha}\geq \alpha$, where $\alpha$ is the component we would obtain by the Gram-Schmidt process for the metric $g$, namely $\alpha=g_{ss}(g_{00}g_{ss}-g_{s0}^2)^{-1}$. Then, we can write
 \begin{align*}
 \til{K}(X_s,X_0)&=\til{K}\left(\frac{X_s}{g_{ss}},Y\right)=\til{R}\left(\frac{X_s}{g_{ss}},Y,Y,\frac{X_s}{g_{ss}}\right)\\
                &=\frac{\til{\alpha}^2}{g_{ss}^2}\til{R}(X_s,X_0,X_0,X_s)=\frac{\til{\alpha}^2}{g_{ss}^2}R_{s00s}^*\\
                &+\frac{\til{\alpha}^2}{g_{ss}^2}\left(-\frac{1}{2}\partial_{ss}^2he^hg_{00}-\frac{1}{2}(\partial_sh)^2e^hg_{00}-e^h\partial_sh\partial_sg_{00}\right).
 \end{align*}
By the observations above and since $R_{s00s}^*\leq R_{s00s}<0$, we get
 \[
 \frac{\til{\alpha}^2}{g_{ss}^2}R_{s00s}^*\leq \frac{\alpha^2}{g_{ss}^2}R_{s00s}=K(X_s,X_0)\leq- \frac{1}{4}.
 \]

 On the other hand, one has
 \begin{align*}
     -\frac{1}{2}\partial_{ss}^2he^hg_{00}-\frac{1}{2}(\partial_sh)^2e^hg_{00}-e^h\partial_sh\partial_sg_{00}\leq-\frac{1}{2}\partial_{ss}^2h-e^h\partial_sh\partial_sg_{00},
 \end{align*}
For this inequality, we used the Taylor expansion of $g_{00}$ in Fermi coordinates (cf. \cite{MANASSE63fermi} after adjusting the indices), which in our case is given by
 \[
 g_{00}=1-2\sum_{k}R_{k00k}|_\gamma x_k^2+\mc{O}(x^3).
 \]
 Thus, for $\ve>0$ small enough we guarantee that $e^hg_{00}\leq1$. This expansion also gives us that $\partial_sg_{00} = \frac{1}{2}x_s+\mc{O}(x_s^2)$, thus the term of order 1 determines the sign of $\partial_sg_{00}$. Hence, we need to study the sign of the function
 \[
 g(x):=-\frac{1}{2}\partial_{ss}^2h-e^h\partial_sh\partial_sg_{00}=\left(f''(x_s)+x_sf'(x_s)e^h\right)\Phi(x).
 \]
  In fact, it is enough to analyze the behavior of $p(x_s)=f''(x_s)+x_sf'(x_s)$, once the delicate part is the analysis for $x_s$ around $0$ where $x_sf'(x_s)>0$, so since $e^h\leq 1$ the worst scenario is given by $p$. Notice that 
 \begin{enumerate}
     \item $p(0)=\frac{1}{4}$.
     \item $p'(0)=0$.
     \item $p''(0)=-\frac{12}{\ve^4}+\frac{1}{2}<0$.
 \end{enumerate}
 We conclude that $x=0$ is a point of maximum value for $p$, indeed it is a global maximum by checking other critical points. See its graph in Figure \ref{fig:Figure2} below.

 \begin{figure}[!htb]
     \centering
     \includegraphics[width=0.7\linewidth]{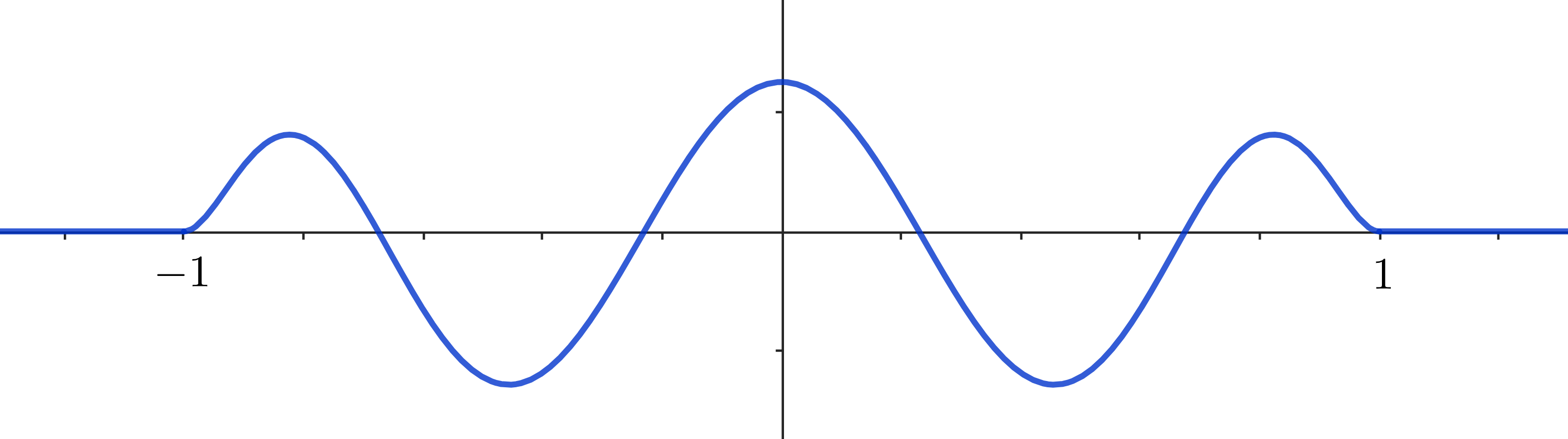}
     \caption{Graph  of  the function $p$ for $\ve=1$}
     \label{fig:Figure2}
 \end{figure}


We conclude that $|g(x)|\leq \frac{1}{4}$ and equality holds if and \noindent only if $x=0$. Finally, $\frac{\til{\alpha}^2}{g_{ss}^2}|g(x)|<\frac{1}{4}$. Then, $\til{K}(X_s,X_0)\leq 0$ with equality if and only if $x=0$.  

\noindent More generally, for any other orthonormal (to $X_s$) vector field $Y$ not equal to $X_0$ we write for $x\neq 0$
\begin{align*}
    \til{K}\left(\frac{X_s}{g_{ss}},Y\right)&=
\frac{1}{g_{ss}^2}\til{R}(X_s,Y,Y,X_s)\\
                &=\frac{1}{g_{ss}^2}R^*(X_s,Y,Y,X_s)+Y_0^2\left(-\frac{1}{2}\partial_{ss}^2he^hg_{00}-\frac{1}{2}(\partial_sh)^2e^hg_{00}-e^h\partial_sh\partial_sg_{00}\right)\\
    &< K(X_s,Y)+\frac{Y_0^2}{4}\\
    &<0.
\end{align*}
\noindent We have proved the following proposition:
\begin{proposition}
    All sectional curvatures for $\til{g}$ are negative, except the sectional curvature $\til{K}(\gamma',e_s)\equiv 0$.
\end{proposition}
\noindent From this proposition, we conclude that $(M,\til{g})$ has non-positive sectional curvatures and just one plane with zero curvature along a single geodesic. We prove the following corollaries that imply Corollary \ref{new geodesic flow has all the good properties}:
\begin{corollary}
      The metric $\til{g}$ has no conjugate points.  
\end{corollary}
\begin{proof}
    This is immediate from non-positive curvature.
\end{proof}
\begin{corollary}
    The geodesic flow $\til{g}_t$ is expansive.
\end{corollary}
\begin{proof}
    If the geodesic flow was not expansive, then by the well-known Flat Strip Theorem \cite{EBERLEIN&ONEIL73visibility} there would exist flat strips on the universal cover of $(M,\til{g})$, which is impossible since $(M,\til{g})$ has only one closed geodesic with zero curvature and negative curvature elsewhere.
\end{proof}
\begin{corollary}
    The dynamical system $(S(T^1M),\tilde{g_t},\widetilde{\mrm{Liou})}$ is ergodic.
\end{corollary}
\begin{proof}
    The set of vectors with a rank greater than one has measure zero, so it follows from the arguments in \cite{BALMANN&BRIN1982ergodicity} and \cite{BURNS1983hyperbolic}.
\end{proof}

\begin{corollary}
    The geodesic flow $\tilde g_t$ has a unique measure of maximal entropy.
\end{corollary}
\begin{proof}
    By Theorem 7.2 of $\cite{CARNEIRO&PUJALS14}$, we know that the metric $\tilde g$ is rank $1$. Now, since $\tilde K\leq 0$, then the main theorem from \cite{KNIEPER1998uniqueness} implies that $\tilde g_t$ admits a unique measure of maximal entropy.
\end{proof}

{\textbf{Acknowledgements:} The authors express their gratitude to Prof. Enrique Pujals for the insightful discussions during the course of this work. The first author thanks his PhD advisor Prof. Gabriel Ponce for his guidance and for presenting the problem. He also thanks Penn State’s Department of Mathematics for their hospitality during the development of this work. In particular, he is tremendously grateful to Prof. Federico Rodriguez-Hertz with whom he had many productive discussions. The second author, supported by Instituto Serrapilheira, grant F0265 “Jangada Dinâmica: Impulsionando Sistemas Dinâmicos na Região Nordeste”, wants to thank the Universidade Federal do Ceará, Brazil, for their hospitality along this project. He also wants to thank Rafael Potrie and Martín Sambarino for many discussions and fruitful conversations. The third author also thanks the Department of Mathematics at SUSTech, China, for its warm hospitality throughout this project.}

\addcontentsline{toc}{section}{References} 
\bibliographystyle{plain}
\bibliography{references}

@inproceedings{CARNEIRO&PUJALS14,
  title={Partially hyperbolic geodesic flows},
  author={Carneiro, F. and Pujals, E.},
  booktitle={Annales de l'IHP Analyse non lin{\'e}aire},
  volume={31(5)},
  pages={985--1014},
  year={2014}
}

@book{WALSCHAP04,
  title={Metric Structures},
  author={Walschap, G.},
  journal={Metric Structures in Differential Geometry},
  pages={131--176},
  year={2004},
  publisher={Springer}
}

@inproceedings{Anosov,
  title={Geodesic flows on closed {R}iemannian manifolds with negative curvature},
  author={Anosov, D.},
  booktitle={Proc. Steklov Inst. Math.},
  volume={90},
  pages={ },
  year={1967}
}

@inproceedings{Donnay,
  title={Anosov geodesic flows for embedded surfaces},
  author={Donnay, V. J. and Pugh, C.},
  booktitle={Asterisque},
  volume={287},
  pages={61--69 },
  year={2003}
}

@article{EBERLEIN73,
  title={When is a geodesic flow of {A}nosov type? {I}},
  author={Eberlein, P.},
  journal={Journal of Differential Geometry},
  volume={8},
  number={3},
  pages={437--463},
  year={1973},
  publisher={Lehigh University}
}

@book{JOST08,
  title={Riemannian geometry and geometric analysis},
  author={Jost, J.},
  year={2008},
  publisher={Springer-Verlag, Berlin},
  pages={xiv+583},
  ISBN = {978-3-540-77340-5},
}

@book{Paternain,
  title={Geodesic Flows},
  author={Paternain, G.},
  volume={180},
  year={1999},
  publisher={Progress in Mathematics}
}

@article{MANASSE63fermi,
  title={Fermi normal coordinates and some basic concepts in differential geometry},
  author={Manasse, F. K. and Misner, C. W.},
  journal={Journal of Mathematical Physics},
  volume={4},
  number={6},
  pages={735--745},
  year={1963},
  publisher={American Institute of Physics}
}

@article{EBERLEIN&ONEIL73visibility,
  title={Visibility manifolds},
  author={Eberlein, P. and O’Neill, B.},
  journal={Pacific Journal of Mathematics},
  volume={46},
  number={1},
  pages={45--109},
  year={1973},
  publisher={Mathematical Sciences Publishers}
}

@book{GOLDMAN1999complex,
  title={Complex hyperbolic geometry},
  author={Goldman, W. M.},
  year={1999},
  publisher={Oxford University Press}
}

@book{BESSE2007einstein,
  title={Einstein manifolds},
  author={Besse, A. L.},
  year={2007},
  publisher={Springer}
}

@article{BALMANN&BRIN1982ergodicity,
  title={On the ergodicity of geodesic flows},
  author={Ballmann, W. and Brin, M.},
  journal={Ergodic Theory and Dynamical Systems},
  volume={2},
  number={3-4},
  pages={311--315},
  year={1982},
  publisher={Cambridge University Press}
}

@article{BURNS1983hyperbolic,
  title={Hyperbolic behaviour of geodesic flows on manifolds with no focal points},
  author={Burns, K.},
  journal={Ergodic Theory and Dynamical Systems},
  volume={3},
  number={1},
  pages={1--12},
  year={1983},
  publisher={Cambridge University Press}
}

@article{CONTRERAS2002partially,
  title={Partially hyperbolic geodesic flows are {A}nosov},
  author={Contreras, G.},
  journal={Comptes Rendus Mathematique},
  volume={334},
  number={7},
  pages={585--590},
  year={2002},
  publisher={Elsevier}
}

@article{RUGGIERO1991persistently,
  title={Persistently expansive geodesic flows},
  author={Ruggiero, R. O.},
  journal={Communications in Mathematical Physics},
  volume={140},
  pages={203--215},
  year={1991},
  publisher={Springer}
}

@article{KNIEPER1998uniqueness,
  title={The uniqueness of the measure of maximal entropy for geodesic flows on rank 1 manifolds},
  author={Knieper, G.},
  journal={Annals of Mathematics},
  pages={291--314},
  year={1998},
  publisher={JSTOR}
}

@article{PESIN1977geodesic,
  title={Geodesic flows on closed Riemannian manifolds without focal points},
  author={Pesin, Ja. B.},
  journal={Mathematics of the USSR-Izvestiya},
  volume={11},
  number={6},
  pages={1195},
  year={1977},
  publisher={IOP Publishing}
}

@article{BURNS1989S2,
  title={Real analytic {B}ernoulli geodesic flows on $\mathbb{S}^2$},
  author={Burns, K. and Gerber, M.},
  journal={{Ergodic Theory and Dynamical Systems}},
  volume={9},
  number={1},
  pages={27--45},
  year={1989},
  publisher={Cambridge University Press}
}

@incollection{DONNAY2006ERGODICITY,
  title={Geodesic flow on the two-sphere part {II}: Ergodicity},
  author={Donnay, V. J.},
  booktitle={Dynamical Systems: Proceedings of the Special Year held at the University of Maryland, College Park, 1986--87},
  pages={112--153},
  year={2006},
  publisher={Springer}
}

@article{GULLIVER1975,
  title={On the variety of manifolds without conjugate points},
  author={Gulliver, R.},
  journal={Transactions of the American Mathematical Society},
  volume={210},
  pages={185--201},
  year={1975}
}

@article{BURNS&WILKINSON2010,
  title={On the ergodicity of partially hyperbolic systems},
  author={Burns, K. and Wilkinson, A.},
  journal={Annals of Mathematics},
  pages={451--489},
  year={2010},
  publisher={JSTOR}
}

@article{BURNS&RH2008,
  title={Density of accessibility for partially hyperbolic diffeomorphisms with one-dimensional center},
  author={Burns, K. and Rodriguez-Hertz, F.  and Rodriguez-Hertz, M.A. and Talitskaya, A. and Ures, R.},
  journal={Discrete Contin. Dyn. Syst},
  volume={22},
  number={1-2},
  pages={75--88},
  year={2008},
  publisher={Citeseer}
}

@article{RHTU2011CRITERIA,
author = {Rodriguez-Hertz, F. and Rodriguez-Hertz, M.A. and Tahzibi, A. and Ures, R.},
title = {{New criteria for ergodicity and nonuniform hyperbolicity}},
volume = {160},
journal = {Duke Mathematical Journal},
number = {3},
publisher = {Duke University Press},
pages = {599 -- 629},
year = {2011},
doi = {10.1215/00127094-1444314},
URL = {https://doi.org/10.1215/00127094-1444314}
}

@article{LIMA&MATHEUS2024,
  title={Polynomial decay of correlations for nonpositively curved surfaces},
  author={Lima, Y. and Matheus, C. and Melbourne, I.},
  journal={Transactions of the American Mathematical Society},
  volume={377},
  number={09},
  pages={6043--6095},
  year={2024}
}

@article{RUGGIERO1991creation,
  title={On the creation of conjugate points},
  author={Ruggiero, R. O.},
  journal={Mathematische Zeitschrift},
  volume={208},
  pages={41--55},
  year={1991},
  publisher={Springer}
}

@article{PATERNAIN1993expansive,
  title={Expansive geodesic flows on surfaces},
  author={Paternain, M.},
  journal={Ergodic Theory and Dynamical Systems},
  volume={13},
  number={1},
  pages={153--165},
  year={1993},
  publisher={Cambridge University Press}
}

@article{RUGGIERO1997expansive,
  title={Expansive geodesic flows in manifolds with no conjugate points},
  author={Ruggiero, R. O.},
  journal={Ergodic Theory and Dynamical Systems},
  volume={17},
  number={1},
  pages={211--225},
  year={1997},
  publisher={Cambridge University Press}
}

@inproceedings{BARTHELME&ERCHENKO2021,
  title={Geometry and entropies in a fixed conformal class on surfaces},
  author={Barthelm{\'e}, T. and Erchenko, A.},
  booktitle={Annales de l'Institut Fourier},
  volume={71},
  pages={731--755},
  year={2021}
}

@article{KATOK1988CONFORMAL,
  title={Four applications of conformal equivalence to geometry and dynamics},
  author={Katok, A.},
  journal={Ergodic Theory Dynam. Systems},
  volume={8},
  number={Charles Conley Memorial Issue},
  pages={139--152},
  year={1988}
}

@article{KATOK1982ENTROPY,
  title={Entropy and closed geodesics},
  author={Katok, A.},
  journal={Ergodic Theory and Dynamical Systems},
  volume={2},
  number={3-4},
  pages={339--365},
  year={1982},
  publisher={Cambridge University Press}
}

@article{FISHER&HASSELBLATT2022accessibility,
  title={Accessibility and centralizers for partially hyperbolic flows},
  author={Fisher, T. and Hasselblatt, B.},
  journal={Ergodic Theory and Dynamical Systems},
  volume={42},
  number={3},
  pages={835--854},
  year={2022},
  publisher={Cambridge University Press}
}

@article {KKM91,
    AUTHOR = {Katok, A. and Knieper, G. and Weiss, H.},
     TITLE = {Formulas for the derivative and critical points of topological entropy for {A}nosov and geodesic flows},
   JOURNAL = {Comm. Math. Phys.},
  FJOURNAL = {Communications in Mathematical Physics},
    VOLUME = {138},
      YEAR = {1991},
    NUMBER = {1},
     PAGES = {19--31},
      ISSN = {0010-3616,1432-0916},
   MRCLASS = {58F17 (58F11 58F15)},
  MRNUMBER = {1108034},
MRREVIEWER = {Donal\ Hurley},
       URL = {http://projecteuclid.org/euclid.cmp/1104202845},
}

@book{LEE18riemannian,
  title={Introduction to Riemannian manifolds},
  author={Lee, J. M.},
  volume={2},
  year={2018},
  publisher={Springer}
}

@article{BESSON95,
  title={Entropies et rigidit{\'e}s des espaces localement sym{\'e}triques de courbure strictement n{\'e}gative},
  author={Besson, G. and Courtois, G. and Gallot, S.},
  journal={Geometric \& Functional Analysis},
  volume={5},
  pages={731--799},
  year={1995},
  publisher={Birkh{\"a}user-Verlag}
}
\Addresses
\end{document}